\documentclass{amsart}

\usepackage[all]{xy}

\usepackage{fixmath}

\newtheorem{lemma}{Lemma}
\newtheorem{theorem}{Theorem}
\newtheorem{proposition}{Proposition}
\newtheorem{corollary}{Corollary}
\theoremstyle{definition}
\newtheorem{definition}{Definition}
\newtheorem{notation}{Notation}
\newtheorem{remark}{Remark}
\newtheorem{example}{Example}

\newcommand{\courbe}[7]{\qbezier(#1,#3)(#4,#7)(#5,#7)
\qbezier(#5,#7)(#6,#7)(#2,#3)}

\newcommand{\NN}{\mathbb{N}^*}

\newcommand{\bs}[1]{\boldsymbol{#1}}

\begin{document}

\title[Orbit duality in ind-varieties of generalized flags]{Orbit duality in ind-varieties of maximal generalized flags}
\author{Lucas Fresse}
\address{Universit\'e de Lorraine, CNRS, Institut \'Elie Cartan de Lorraine, UMR 7502, Van\-doeuvre-l\`es-Nancy, F-54506 France}
\email{lucas.fresse@univ-lorraine.fr}
\author{Ivan Penkov}
\address{Jacobs University Bremen, Campus Ring 1, 28759 Bremen, Germany}
\email{i.penkov@jacobs-university.de}

\dedicatory{To Ernest Borisovich Vinberg \\ on the occasion of his 80th birthday}

\begin{abstract}
We extend Matsuki duality to arbitrary ind-varieties of maximal generalized flags, in other words, to any homogeneous ind-variety $\mathbf{G}/\mathbf{B}$ for a classical ind-group $\mathbf{G}$ and a splitting Borel ind-subgroup $\mathbf{B}\subset\mathbf{G}$.  As a first step, we present an explicit combinatorial version of Matsuki duality in the finite-dimensional case, involving an explicit parametrization of $K$- and $G^0$-orbits on $G/B$.  After proving Matsuki duality in the infinite-dimensional case, we give necessary and sufficient conditions on a Borel ind-subgroup $\mathbf{B}\subset\mathbf{G}$ for the existence of open and closed $\mathbf{K}$- and $\mathbf{G}^0$-orbits on $\mathbf{G}/\mathbf{B}$, where $\left(\mathbf{K},\mathbf{G}^0\right)$ is an aligned pair of a symmetric ind-subgroup $\mathbf{K}$ and a real form $\mathbf{G}^0$ of $\mathbf{G}$.
\end{abstract}

\keywords{Classical ind-groups; generalized flags; symmetric pairs; real forms; Matsuki duality}
\subjclass[2010]{14L30; 14M15; 22E65; 22F30}

\maketitle

\section{Introduction}

\label{S.intro}

In this paper we extend Matsuki duality to ind-varieties of maximal generalized flags, i.e., to homogeneous ind-spaces of the form $\mathbf{G}/\mathbf{B}$ for $\mathbf{G}=\mathrm{GL}(\infty)$, $\mathrm{SL}(\infty)$, $\mathrm{SO}(\infty)$, $\mathrm{Sp}(\infty)$.  In the case of a finite-dimensional reductive algebraic group $G$, Matsuki duality \cite{Gindikin-Matsuki, Matsuki, Matsuki3} is a bijection between the (finite) set of $K$-orbits on $G/B$ and the set of $G^0$-orbits on $G/B$, where $K$ is a symmetric subgroup of $G$ and $G^0$ is a real form of $G$.  Moreover, this bijection reverses the inclusion relation between orbit closures.  In particular, the remarkable theorem about the uniqueness of a closed $G^0$-orbit on $G/B$, see \cite{W1}, follows via Matsuki duality from the uniqueness of a (Zariski) open $K$-orbit on $G/B$.
In the monograph \cite{FHW}, Matsuki duality has been used as the starting point in a study of cycle spaces.

If $\mathbf{G}=\mathrm{GL}(\infty)$, $\mathrm{SL}(\infty)$, $\mathrm{SO}(\infty)$, $\mathrm{Sp}(\infty)$ is a classical ind-group, then its Borel ind-subgroups are neither $\mathbf{G}$-conjugate nor $\mathrm{Aut}(\mathbf{G})$-conjugate, hence there are many ind-varieties of the form $\mathbf{G}/\mathbf{B}$.  We show that Matsuki duality extends to any ind-variety $\mathbf{G}/\mathbf{B}$ where $\mathbf{B}$ is a splitting Borel ind-subgroup of $\mathbf{G}$ for $\mathbf{G}=\mathrm{GL}(\infty)$, $\mathrm{SL}(\infty)$, $\mathrm{SO}(\infty)$, $\mathrm{Sp}(\infty)$.  In the infinite-dimensional case, the structure of $\mathbf{G}^0$-orbits and $\mathbf{K}$-orbits on $\mathbf{G}/\mathbf{B}$ is more complicated than in the finite-dimensional case, and there are always infinitely many orbits.  

A first study of the $\mathbf{G}^0$-orbits on $\mathbf{G}/\mathbf{B}$ for $\mathbf{G}=\mathrm{GL}(\infty),\mathrm{SL}(\infty)$ was done in~\cite{Ignatyev-Penkov-Wolf} and was continued in~\cite{W}.  In particular, in \cite{Ignatyev-Penkov-Wolf} it was shown that, for some real forms $\mathbf{G}^0$, there are splitting Borel ind-subgroups $\mathbf{B}\subset \mathbf{G}$ such that $\mathbf{G}/\mathbf{B}$ has neither an open nor a closed $\mathbf{G}^0$-orbit.  
We know of no prior studies of the structure of $\mathbf{K}$-orbits on $\mathbf{G}/\mathbf{B}$ of $\mathbf{G}=\mathrm{GL}(\infty),\mathrm{SL}(\infty),\mathrm{SO}(\infty),\mathrm{Sp}(\infty)$.  The duality we establish in this paper shows that the structure of $\mathbf{K}$-orbits on $\mathbf{G}/\mathbf{B}$ is a ``mirror image'' of the structure of $\mathbf{G}^0$-orbits on $\mathbf{G}/\mathbf{B}$.  In particular, the fact that $\mathbf{G}/\mathbf{B}$ admits at most one closed $\mathbf{G}^0$-orbit is now a corollary of the obvious statement that $\mathbf{G}/\mathbf{B}$ admits at most one Zariski-open $\mathbf{K}$-orbit.  


Our main result can be stated as follows.
Let $(\mathbf{G},\mathbf{K},\mathbf{G}^0)$ be one of the triples 
listed in Section \ref{S2.1}
consisting of
a classical (complex) ind-group $\mathbf{G}$, a symmetric ind-subgroup $\mathbf{K}\subset\mathbf{G}$,
and the corresponding real form $\mathbf{G}^0\subset\mathbf{G}$.
Let $\mathbf{B}\subset\mathbf{G}$ be a splitting Borel ind-subgroup
such that $\mathbf{X}:=\mathbf{G}/\mathbf{B}$ is 
an ind-variety of maximal generalized flags (isotropic, in types B, C, D)
weakly compatible with a basis adapted to 
the choice of $\mathbf{K}$, $\mathbf{G}^0$ in the sense of Sections \ref{S2.1}, \ref{S2.3}.
There are natural exhaustions $\mathbf{G}=\bigcup_{n\geq 1}G_n$
and $\mathbf{X}=\bigcup_{n\geq 1}X_n$.
Here $G_n$ is a finite-dimensional algebraic group,
$X_n$ is the full flag variety of $G_n$,
and the inclusion $X_n\subset\mathbf{X}$ is in particular $G_n$-equivariant.
Moreover $K_n:=\mathbf{K}\cap G_n$ and $G^0_n:=\mathbf{G}^0\cap G_n$
are respectively a symmetric subgroup and the corresponding real form of $G_n$.
See Section \ref{S4.4} for more details.

\begin{theorem}
\label{theorem-1}
\begin{itemize}
\item[\rm (a)] For every $n\geq 1$ the inclusion $X_n\subset \mathbf{X}$
induces embeddings of orbit sets $X_n/K_n\hookrightarrow \mathbf{X}/\mathbf{K}$
and $X_n/G^0_n\hookrightarrow \mathbf{X}/\mathbf{G}^0$.
\item[\rm (b)] There is a bijection $\Xi:\mathbf{X}/\mathbf{K}\to\mathbf{X}/\mathbf{G}^0$
such that the diagram
\[
\xymatrix{
X_n/K_n \ar@{->}[d]^{\Xi_n} \ar@{^{(}->}[r] & \mathbf{X}/\mathbf{K} \ar@{->}[d]^{\Xi} \\
X_n/G^0_n \ar@{^{(}->}[r] & \mathbf{X}/\mathbf{G}^0
}
\]
is commutative, where $\Xi_n$ stands for Matsuki duality.
\item[\rm (c)] For every $\mathbf{K}$-orbit $\bs{\mathcal{O}}\subset\mathbf{X}$
the intersection $\bs{\mathcal{O}}\cap\Xi(\bs{\mathcal{O}})$ consists of a single $\mathbf{K}\cap\mathbf{G}^0$-orbit.
\item[\rm (d)] The bijection $\Xi$ reverses the inclusion relation of orbit closures.
In particular $\Xi$ maps open (resp., closed) $\mathbf{K}$-orbits to
closed (resp., open) $\mathbf{G}^0$-orbits.
\end{itemize}
\end{theorem}

Actually our results are much more precise: in Propositions \ref{P4-1.1}, \ref{P4.2-2}, \ref{P4.3} we show that $\mathbf{X}/\mathbf{K}$ and $\mathbf{X}/\mathbf{G}^0$ admit the same explicit parametrization which is nothing but the inductive limit of suitable joint parametrizations of $X_n/K_n$ and $X_n/G_n^0$.  This yields the bijection $\Xi$ of Theorem \ref{theorem-1}\,{\rm (b)}. Parts {\rm (a)} and {\rm (b)} of Theorem \ref{theorem-1} are
implied by our claims (\ref{T-proof.1}), (\ref{T-proof.2}), (\ref{T-proof.3}) below. Theorem \ref{theorem-1}\,{\rm (c)} follows from the corresponding statements in Propositions \ref{P4-1.1}, \ref{P4.2-2}, \ref{P4.3}.
Finally, Theorem \ref{theorem-1}\,{\rm (d)} is implied by 
Theorem \ref{theorem-1}\,{\rm (a)}--{\rm (b)}, the definition of the ind-topology, 
and the fact that the duality $\Xi_n$ reverses the inclusion relation between orbit closures.

\subsection*{Organization of the paper}

In Section \ref{S2} we introduce the notation for classical ind-groups, symmetric ind-subgroups, and real forms.
We recall some basic facts on finite-dimensional flag varieties, as well as the notion of ind-variety of generalized flags \cite{Dimitrov-Penkov,Ignatyev-Penkov}.
In Section \ref{S3} we give 
the joint parametrization of $K$- and $G^0$-orbits in a finite-dimensional flag variety.
This parametrization should be known in principle (see \cite{MO,Y})
but we have not found a reference where it would appear exactly as we present it. For the sake of completeness we provide full proofs of these results.
In
Section \ref{S4} we state our main results on the parametrization of $\mathbf{K}$- and $\mathbf{G}^0$-orbits in ind-varieties of generalized flags. Theorem \ref{theorem-1} above is a consequence of these results.
In Section \ref{S5} we point out some further corollaries of our main results.

In what follows $\NN$ stands for the set of positive integers.
$|A|$ stands for the cardinality of a set $A$.  The symmetric group on $n$ letters is denoted by $\mathfrak{S}_n$ and $\mathfrak{S}_\infty=\lim\limits_{\longrightarrow}\,\mathfrak{S}_n$ stands for the infinite symmetric group. 
Often we write $w_k$ for the image $w(k)$ of $k$ by a permutation $w$.
By $(k;\ell)$ we denote the transposition that switches $k$ and $\ell$.
We use boldface letters to denote ind-varieties.  
An index of notation can be found at the end of the paper.

\subsection*{Acknowledgement}

We thank Alan Huckleberry and Mikhail Ignatyev for their encouragement to study Matsuki duality.  The first author was supported in part by ISF Grant Nr. 797/14 and by ANR project GeoLie (ANR-15-CE40-0012).  The second author was supported in part by DFG Grant PE 980/6-1.

\section{Notation and preliminary facts}

\label{S2}

\subsection{Classical groups and classical ind-groups}

\label{S2.1}

Let $\mathbf{V}$ be a complex vector space of countable dimension, with a basis
$E=(e_1,e_2,\ldots)=(e_\ell)_{\ell\in \NN}$.  
Every vector $x\in\mathbf{V}$ is identified with the column of its coordinates in the basis $E$,
and $x\mapsto \overline{x}$ stands for complex conjugation with respect to $E$. 
We also consider the finite dimensional subspace $V=V_n:=\langle e_1,\ldots,e_n\rangle_\mathbb{C}$ of $\mathrm{V}$.

The classical ind-group $\mathrm{GL}(\infty)$ is defined as
\[\mathrm{GL}(\infty)=\mathbf{G}(E):=\{g\in\mathrm{Aut}(\mathbf{V}):g(e_\ell)=e_\ell\mbox{ for all $\ell\gg 1$}\}=\bigcup_{n\geq 1}\mathrm{GL}(V_n).\]
%
The real forms of $\mathrm{GL}(\infty)$ are well known and can be traced back to the work of Baranov~\cite{Baranov}.  Below we list aligned pairs $(\mathbf{K},\mathbf{G}^0)$, where $\textbf{G}^0$ is a real form of $\mathbf{G}$ and $\mathbf{K}\subset\mathbf{G}$ is a symmetric ind-subgroup of $\mathbf{G}$.  The pairs $\left(\mathbf{K},\mathbf{G}^0\right)$ we consider are aligned in the following way: there exists an exhaustion of $\mathbf{G}$ as a union $\bigcup_n\,\mathrm{GL}\left(V_n\right)$ such that $K_n:=\mathbf{K}\cap\mathrm{GL}\left(V_n\right)$ is a symmetric subgroup of $\mathrm{GL}\left(V_n\right)$, $G_0^n:=\mathbf{G}^0\cap\mathrm{GL}\left(V_n\right)$ is a real form of $\mathrm{GL}\left(V_n\right)$, and $K_n\cap G_n^0$ is a maximal compact subgroup of $G_n^0$.

\subsubsection{Types A1 and A2}
\label{S:A1A2}

Let $\Omega$ be a $\NN\times\NN$-matrix of the form
\begin{equation}
\label{omega}
\Omega=\left(
\begin{array}{cccc}
J_1 & & (0) \\
 & J_2 \\
(0) & & \ddots
\end{array}
\right)\ 
\mbox{where}\ 
\left\{
\begin{array}{ll}
J_k\in\left\{\left(\begin{matrix} 0 & 1 \\ 1 & 0 \end{matrix}\right),\left(1\right)\right\}
& \parbox{2.6cm}{(orthogonal case, \\ type A1),} \\[4mm]
J_k=\left(\begin{matrix} 0 & 1 \\ -1 & 0 \end{matrix}\right) & \parbox{2.6cm}{(symplectic case, \\ type A2).}
\end{array}
\right.
\end{equation}
The bilinear form
\[\omega(x,y):={}^tx\Omega y\quad (x,y\in \mathbf{V})\]
is symmetric in type A1 and symplectic in type A2, whereas the map
\[\gamma(x):=\Omega\overline{x}\quad (x\in \mathbf{V})\]
is an involution of $\mathbf{V}$ in type A1 and an antiinvolution in type A2.
Let
\begin{eqnarray*}
 & & \mathbf{K}=\mathbf{G}(E,\omega):=\{g\in \mathbf{G}(E):\omega(gx,gy)=\omega(x,y)\ \forall x,y\in\mathbf{V}\} \\[1mm]
 & \mbox{and} & \mathbf{G}^0:=\{g\in \mathbf{G}(E):\gamma(gx)=g\gamma(x)\ \forall x\in\mathbf{V}\}.
\end{eqnarray*}

\subsubsection{Type A3}
\label{S:A3}
Fix a (proper) decomposition $\NN=N_+\sqcup N_-$ and let
\begin{equation}
\label{phi}
\Phi=\left(
\begin{array}{cccc}
\epsilon_1 & & (0) \\
 & \epsilon_2 \\
(0) & & \ddots
\end{array}
\right)
\end{equation}
where $\epsilon_\ell=1$ for $\ell\in N_+$ and $\epsilon_\ell=-1$ for $\ell\in N_-$.
Thus
\[\phi(x,y):={}^t\overline{x} \Phi y\quad(x,y\in\mathbf{V})\]
is a Hermitian form of signature $(|N_+|,|N_-|)$ and
\[\delta(x):=\Phi x\quad(x\in\mathbf{V})\]
is an involution. Finally let
\[
\mathbf{K}:=\{g\in \mathbf{G}(E):\delta(gx)=g\delta(x)\ \forall
x\in\mathbf{V}\}
\]
and
\[
\mathbf{G}^0:=\{g\in \mathbf{G}(E):\phi(gx,gy)=\phi(x,y)\ \forall
 x,y\in\mathbf{V}\}.
\]



\subsection*{Types B, C, D}
Next we describe pairs $(\mathbf{K},\mathbf{G}^0)$ associated to
the other classical ind-groups $\mathrm{SO}(\infty)$ and $\mathrm{Sp}(\infty)$.  Let $\mathbf{G}=\mathbf{G}(E,\omega)$
where $\omega$ is a (symmetric or symplectic) bilinear form given by a matrix $\Omega$ as in (\ref{omega}).
In view of (\ref{omega}), for every $\ell\in\NN$ there is a unique $\ell^*\in\NN$ such that
\[\omega(e_\ell,e_{\ell^*})\not=0.\]
Moreover $\ell^*\in\{\ell-1,\ell,\ell+1\}$. The map $\ell\mapsto\ell^*$ is an involution of $\NN$.


\subsubsection{Types BD1 and C2}
\label{S:BD1C2}

Assume that $\omega$ is symmetric in type BD1 and symplectic in type C2.
Fix a (proper) decomposition $\NN=N_+\sqcup N_-$ such that
\[\forall \ell\in\NN,\ \ell\in N_+\Leftrightarrow \ell^*\in N_+\]
and the restriction of $\omega$ on each of the subspaces $\mathbf{V}_+:=\langle e_\ell:\ell\in N_+\rangle_\mathbb{C}$ and
$\mathbf{V}_-:=\langle e_\ell:\ell\in N_-\rangle_\mathbb{C}$ is nondegenerate.
Let $\Phi,\phi,\delta$ be as in Section~\ref{S:A3}.  Then we set
\begin{eqnarray}
\label{N1}
 \mathbf{K}:=\{g\in \mathbf{G}(E,\omega):\delta(gx)=g\delta(x)\ \forall
x\in\mathbf{V}\}
\end{eqnarray}
and
\begin{eqnarray}
\label{N2}
 \mathbf{G}^0:=\{g\in \mathbf{G}(E,\omega):\phi(gx,gy)=\phi(x,y)\ \forall
 x,y\in\mathbf{V}\}.
\end{eqnarray}


\subsubsection{Types C1 and D3}
\label{S:C1D3}
Assume that $\omega$ is symmetric in type D3 and symplectic in type C1.
Fix a decomposition $\NN=N_+\sqcup N_-$ satisfying
\[
\forall \ell\in\NN,\ \ell\in N_+\Leftrightarrow \ell^*\in N_-.
\]
Note that this forces every block $J_k$ in (\ref{omega}) to be of size $2$.
In this situation $\mathbf{V}_+:=\langle e_\ell:\ell\in N_+\rangle_\mathbb{C}$ and $\mathbf{V}_-:=\langle e_\ell:\ell\in N_-\rangle_\mathbb{C}$
are maximal isotropic subspaces for the form $\omega$.
Let $\Phi,\phi,\delta$ be as in Section~\ref{S:A3}. Finally, we define the ind-subgroups $\mathbf{K},\mathbf{G}^0\subset\mathbf{G}$ as in (\ref{N1}), (\ref{N2}).


\subsection*{Finite-dimensional case} 
The following table summarizes the form of the intersections $G=\textbf{G}\cap\mathrm{GL}\left(V_n\right)$, $K=\textbf{K}\cap\mathrm{GL}\left(V_n\right)$, $G^0=\textbf{G}^0\cap\mathrm{GL}\left(V_n\right)$, where $n=2m$ is even whenever we are in types A2, C1, C2, and D3.
In types A3, BD1, and C2, we set $(p,q)=(|N_+\cap\{1,\ldots,n\}|,|N_-\cap\{1,\ldots,n\}|)$.
By $\mathbb{H}$ we denote the skew field of quaternions.
In this way we retrieve the classical finite-dimensional symmetric pairs and real forms
(see, e.g., \cite{Berger,OV,Otha1}).

\begin{center}
\renewcommand{\arraystretch}{1.25}
\begin{tabular}{|c|c|c|c|}
\hline
type & $G:=\mathbf{G}\cap\mathrm{GL}\left(V_n\right)$ & $K:=\mathbf{K}\cap\mathrm{GL}\left(V_n\right)$ & $G^0:=\mathbf{G}^0\cap\mathrm{GL}\left(V_n\right)$ \\ 
\hline
A1 & & $\mathrm{O}_n(\mathbb{C})$ & $\mathrm{GL}_n(\mathbb{R})$ \\
A2 & $\mathrm{GL}_n(\mathbb{C})$ & $\mathrm{Sp}_n(\mathbb{C})$ & $\mathrm{GL}_m(\mathbb{H})$ \\
A3 & & $\mathrm{GL}_p(\mathbb{C})\times\mathrm{GL}_q(\mathbb{C})$ & $\mathrm{U}_{p,q}(\mathbb{C})$ \\
\hline
BD1 & $\mathrm{O}_n(\mathbb{C})$ & $\mathrm{O}_p(\mathbb{C})\times\mathrm{O}_q(\mathbb{C})$ & $\mathrm{O}_{p,q}(\mathbb{C})$ \\
\hline
C1 & ${}_{{}_{\mbox{$\mathrm{Sp}_n(\mathbb{C})$}}}$ & $\mathrm{GL}_m(\mathbb{C})$ & $\mathrm{Sp}_n(\mathbb{R})$ \\[-1.5mm]
C2 & & $\mathrm{Sp}_p(\mathbb{C})\times\mathrm{Sp}_q(\mathbb{C})$ & $\mathrm{Sp}_{p,q}(\mathbb{C})$ \\
\hline
D3 & $\mathrm{O}_n(\mathbb{C})=\mathrm{O}_{2m}(\mathbb{C})$ & $\mathrm{GL}_m(\mathbb{C})$ & $\mathrm{O}^*_n(\mathbb{C})$ \\
\hline
\end{tabular}
\end{center}
In each case $G^0$ is a real form obtained from $K$ so that $K\cap G^0$ is a maximal compact subgroup of $G^0$.
Conversely $K$ is obtained from $G^0$ as the complexification of a maximal compact subgroup.

\subsection{Finite-dimensional flag varieties}

\label{S2.2}

Recall that $V=V_n$.
The flag variety $X:=\mathrm{GL}(V)/B=\{gB:g\in \mathrm{GL}(V)\}$ (for a Borel subgroup $B\subset \mathrm{GL}(V)$)
can as well be viewed as the set of Borel subgroups $\{gBg^{-1}:g\in \mathrm{GL}(V)\}$
or as the set of complete flags
\begin{equation}
\label{2.2.X}
\big\{\mathcal{F}=(F_0\subset F_1\subset\ldots\subset F_n=V): \dim F_k=k\ \mbox{ for all $k$}\big\}.
\end{equation}
For every complete flag $\mathcal{F}$ let $B_\mathcal{F}:=\{g\in \mathrm{GL}(V):g\mathcal{F}=\mathcal{F}\}$ denote the corresponding Borel subgroup.
When $(v_1,\ldots,v_n)$ is a basis of $V$ we write
\[\mathcal{F}(v_1,\ldots,v_n):=\big(0\subset\langle v_1\rangle_\mathbb{C}\subset\langle v_1,v_2\rangle_\mathbb{C}\subset\ldots\subset\langle v_1,\ldots,v_n\rangle_\mathbb{C}\big)\in X.\]

\medskip

\paragraph{{\bf Bruhat decomposition}}
The double flag variety $X\times X$ has a finite number
of $\mathrm{GL}(V)$-orbits parametrized by permutations $w\in\mathfrak{S}_n$.
Specifically, given two flags $\mathcal{F}=(F_k)_{k=0}^n$ and
$\mathcal{F}'=(F'_\ell)_{\ell=0}^n$ there is a unique permutation
$w=:w(\mathcal{F},\mathcal{F}')$ such that
\[\dim F_k\cap F'_\ell=\big|\big\{j\in\{1,\ldots,\ell\}:w_j\in\{1,\ldots,k\}\big\}\big|.\]
The permutation $w(\mathcal{F},\mathcal{F}')$ is called the relative position of $(\mathcal{F},\mathcal{F}')\in X\times X$.
Then
\[X\times X=\bigsqcup_{w\in\mathfrak{S}_n}\mathbb{O}_w
\quad\mbox{where $\mathbb{O}_w:=\big\{(\mathcal{F},\mathcal{F}')\in X\times X:w(\mathcal{F},\mathcal{F}')=w\big\}$}\]
is the decomposition of $X\times X$ into $\mathrm{GL}(V)$-orbits.
The unique closed orbit is $\mathbb{O}_{\mathrm{id}}$ and the unique open
orbit is $\mathbb{O}_{w_0}$ where $w_0$ is the involution given by
$w_0(k)=n-k+1$ for all $k$. The map
$\mathbb{O}_w\mapsto\mathbb{O}_{w_0w}$ is an involution on the set
of orbits and reverses inclusions between orbit closures.
Representatives of $\mathbb{O}_w$ can be obtained as follows:
for every basis $(v_1,\ldots,v_n)$ of $V$ we have
\[\big(\mathcal{F}(v_1,\ldots,v_n),\mathcal{F}(v_{w_1},\ldots,v_{w_n})\big)\in\mathbb{O}_w.\]

\medskip

\paragraph{{\bf Variety of isotropic flags}} Let $V$ be endowed with a nondegenerate symmetric or symplectic bilinear form $\omega$. 
For a subspace $F\subset V$, set $F^\perp=\{x\in V:\omega(x,y)=0\ \forall y\in F\}$.
The variety of isotropic flags is the subvariety $X_\omega$ of $X$, where
\begin{equation}
\label{2.2.Xomega}
X_\omega=\{\mathcal{F}=(F_k)_{k=0}^n\in X:F_k^\perp=F_{n-k}\ \forall k=0,\ldots,n\}.
\end{equation}
It is endowed with a transitive action of the subgroup
$G(V,\omega)\subset\mathrm{GL}(V)$ of automorphisms preserving $\omega$. 

\begin{lemma}
\label{lemma-2.2.1}
{\rm (a)}
For every endomorphism $f\in\mathrm{End}(V)$, let $f^*\in\mathrm{End}(V)$ denote the endomorphism adjoint to $f$ with respect to $\omega$.
Let $H\subset \mathrm{GL}(V)$ be a subgroup satisfying the condition
\begin{equation}
\label{2.2.1}
\mathbb{C}[g^*g]\cap\mathrm{GL}(V)\subset H\ \mbox{ for all $g\in H$}.
\end{equation}
Assume that $\mathcal{F}\in X_\omega$ and $\mathcal{F}'\in X_\omega$ belong to the same $H$-orbit of $X$.  Then they belong to the same $H\cap G(V,\omega)$-orbit of $X_\omega$. \\
{\rm (b)} Let $H=\{g\in \mathrm{GL}(V):g(V_+)=V_+,\ g(V_-)=V_-\}$ where $V=V_+\oplus V_-$ is a decomposition such that
$(V_+^\perp,V_-^\perp)=(V_+,V_-)$ or $(V_-,V_+)$.
Then (\ref{2.2.1}) is fulfilled.
\end{lemma}

\begin{proof}
{\rm (a)}
Note that $G(V,\omega)=\{g\in\mathrm{GL}(V):g^*=g^{-1}\}$.
Consider $g\in H$ such that $\mathcal{F}'=g\mathcal{F}$.
The equality $(gF)^\perp=(g^*)^{-1}F^\perp$ holds for all subspaces $F\subset V$.
Since $\mathcal{F},\mathcal{F}'$ belong to $X_\omega$ we have $\mathcal{F}'=(g^*)^{-1}\mathcal{F}$,
hence $g^*g\mathcal{F}=\mathcal{F}$.
Let $g_1=g^*g$.
By \cite[Lemma 1.5]{Jantzen} there is a polynomial $P(t)\in\mathbb{C}[t]$ such that 
$P(g_1)^2=g_1$. Set $h=P(g_1)$.
Then $h\in\mathrm{GL}(V)$ (since $h^2=g_1\in\mathrm{GL}(V)$), and (\ref{2.2.1}) shows that actually $h\in H$.
Moreover $h^*=h$ (since $h\in\mathbb{C}[g_1]$ and $g_1^*=g_1$)
and $h\mathcal{F}=\mathcal{F}$ (as each subspace in $\mathcal{F}$ is $g_1$-stable hence also $h$-stable).
Set $h_1:=gh^{-1}\in H$.
Then, on the one hand,
\[h_1^*=(h^*)^{-1}g^*=h^{-1}g_1g^{-1}=h^{-1}h^2g^{-1}=hg^{-1}=h_1^{-1}\,.\]
Thus $h_1\in H\cap G(V,\omega)$, 
and on the other hand,
$h_1\mathcal{F}=gh^{-1}\mathcal{F}=g\mathcal{F}=\mathcal{F}'$. \\
{\rm (b)}
The equality $g^*(gF)^\perp=F^\perp$ (already mentioned) applied to $F=V_\pm$ yields $g^*\in H$, and thus $g^*g\in H$, whenever $g\in H$. This implies (\ref{2.2.1}).
\end{proof}

\begin{remark}
The proof of Lemma \ref{lemma-2.2.1}\,{\rm (a)} is inspired by \cite[\S1.4]{Jantzen}. We also refer to \cite{Nishiyama, Otha} for similar results and generalizations.
\end{remark}

\subsection{Ind-varieties of generalized flags}

\label{S2.3}

Recall that $\mathbf{V}$ denotes a complex vector space of countable dimension,
with a basis $E=(e_\ell)_{\ell\in\NN}$.

\begin{definition}[\cite{Dimitrov-Penkov}]
Let $\mathcal{F}$ be a chain of subspaces in $\mathbf{V}$, i.e., a set of subspaces of $\mathbf{V}$ which is totally ordered by inclusion. Let $\mathcal{F}'$ (resp., $\mathcal{F}''$) be the subchain consisting of all $F\in\mathcal{F}$ with an immediate successor (resp., an immediate predecessor).
By $s(F)\in\mathcal{F}''$ we denote the immediate successor of $F\in\mathcal{F}'$.

A {\em generalized flag} in $\mathbf{V}$ is a chain of subspaces $\mathcal{F}$ such that:
\begin{itemize}
\item[\rm (i)] each $F\in\mathcal{F}$ has an immediate successor or predecessor,
i.e., $\mathcal{F}=\mathcal{F}'\cup\mathcal{F}''$;
\item[\rm (ii)] for every $v\in\mathbf{V}\setminus\{0\}$ there is a unique $F_v\in\mathcal{F}'$ such that $v\in s(F_v)\setminus F_v$, i.e.,
$\mathbf{V}\setminus\{0\}=\bigcup_{F\in\mathcal{F}'}(s(F)\setminus F)$.
\end{itemize}

A generalized flag is {\em maximal} if it is not properly contained in another generalized flag.
Specifically, $\mathcal{F}$ is maximal if and only if $\dim s(F)/F=1$ for all $F\in\mathcal{F}'$.
\end{definition}

\begin{notation}
Let $\sigma:\NN\to (A,\prec)$ be a surjective map onto a totally ordered set.
Let $\underline{v}=(v_1,v_2,\ldots)$ be a basis of $\mathbf{V}$.
For every $a\in A$, let 
\[F'_a=\langle v_\ell:\sigma(\ell)\prec a\rangle_\mathbb{C},\quad 
F''_a=\langle v_\ell:\sigma(\ell)\preceq a\rangle_\mathbb{C}.\]
Then $\mathcal{F}=\mathcal{F}_\sigma(\underline{v}):=\{F'_a,F''_a:a\in A\}$ is a generalized flag such that
$\mathcal{F}'=\{F'_a:a\in A\}$, $\mathcal{F}''=\{F''_a:a\in A\}$, and $s(F'_a)=F''_a$ for all $a$.
We call such a generalized flag {\em compatible with the basis $\underline{v}$}.

Moreover, $\mathcal{F}_\sigma(\underline{v})$ is maximal if and only if the map $\sigma$ is bijective.

We use the abbreviation $\mathcal{F}_\sigma:=\mathcal{F}_\sigma(E)$.

Note that every generalized flag has a compatible basis \cite[Proposition 4.1]{Dimitrov-Penkov}.
A generalized flag is {\em weakly compatible with $E$} if it is compatible with some basis $\underline{v}$ such that $E\setminus(E\cap \underline{v})$ is finite (equivalently, $\dim\mathbf{V}/\langle E\cap \underline{v}\rangle_\mathbb{C}<\infty$).
\end{notation}

The group $\mathbf{G}(E)$ (as well as $\mathrm{Aut}(\mathbf{V})$) acts on generalized flags in a natural way.
Let $\mathbf{P}_\mathcal{F}\subset\mathbf{G}(E)$ denote the ind-subgroup of elements preserving $\mathcal{F}$.
It is a closed ind-subgroup of $\mathbf{G}(E)$.
If $\mathcal{F}$ is compatible with $E$, then $\mathbf{P}_\mathcal{F}$ is a splitting parabolic ind-subgroup of $\mathbf{G}(E)$ in the sense that
it is locally parabolic (i.e., there exists an exhaustion of $\mathbf{G}(E)$ by finite-dimensional reductive algebraic subgroups $G_n$ such that the intersections $\mathbf{P}_\mathcal{F}\cap G_n$ are parabolic subgroups of $G_n$) and contains the Cartan ind-subgroup $\mathbf{H}(E)\subset \mathbf{G}(E)$ of elements diagonal in $E$.
Moreover if $\mathcal{F}$ is maximal, then $\mathbf{B}_\mathcal{F}:=\mathbf{P}_\mathcal{F}$ is a splitting Borel ind-subgroup
(i.e., all intersections $\mathbf{B}_\mathcal{F}\cap G_n$ as above are Borel subgroups of $G_n$).

\begin{definition}[\cite{Dimitrov-Penkov}]
Two generalized flags $\mathcal{F},\mathcal{G}$ are called {\em $E$-commensurable}
if $\mathcal{F},\mathcal{G}$ are weakly compatible with $E$, and there is an isomorphism $\phi:\mathcal{F}\to\mathcal{G}$ of ordered sets and a finite dimensional subspace $U\subset\mathbf{V}$ such that
\begin{itemize}
\item[\rm (i)] $\phi(F)+U=F+U$ for all $F\in\mathcal{F}$;
\item[\rm (ii)] $\dim \phi(F)\cap U=\dim F\cap U$ for all $F\in\mathcal{F}$.
\end{itemize}
\end{definition}

$E$-commensurability is an equivalence relation on the set of generalized flags weakly compatible with $E$. In fact, according to the following proposition, each equivalence class consists of a single $\mathbf{G}(E)$-orbit.
If $\mathcal{F}$ is a generalized flag weakly compatible with $E$ we denote by $\mathbf{X}(\mathcal{F},E)$ the set of generalized flags which are $E$-commensurable with $\mathcal{F}$.

\begin{proposition}[\cite{Dimitrov-Penkov}]
\label{P2.3-1}
The set $\mathbf{X}=\mathbf{X}(\mathcal{F},E)$ is endowed with a natural structure of ind-variety.
Moreover $\mathbf{X}$ is $\mathbf{G}(E)$-homogeneous and the map $g\mapsto g\mathcal{F}$ induces an isomorphism of ind-varieties $\mathbf{G}(E)/\mathbf{P}_\mathcal{F}\stackrel{\sim}{\to}\mathbf{X}$.
\end{proposition}


\begin{proposition}[\cite{Fresse-Penkov}]
\label{P2-3.2}
Let $\sigma:\NN\to (A,\prec)$ and $\tau:\NN\to (B,\prec)$ be maps onto two totally ordered sets.
\begin{itemize}
\item[\rm (a)] Each $E$-compatible generalized flag in $\mathbf{X}(\mathcal{F}_\sigma,E)$ is of the form $\mathcal{F}_{\sigma w}$ for $w\in \mathfrak{S}_\infty$.
Moreover $\mathcal{F}_{\sigma w}=\mathcal{F}_{\sigma w'}\Leftrightarrow
w'w^{-1}\in\mathrm{Stab}_\sigma:=\{v\in\mathfrak{S}_\infty:\sigma v=\sigma\}$.
\item[\rm (b)] Assume that $\mathcal{F}_\tau$ is maximal (i.e., $\tau$ is bijective)
so that $\mathbf{B}_{\mathcal{F}_\tau}$ is a splitting Borel ind-subgroup.
Then each $\mathbf{B}_{\mathcal{F}_\tau}$-orbit of $\mathbf{X}(\mathcal{F}_\sigma,E)$ 
contains a unique element of the form $\mathcal{F}_{\sigma w}$ for $w\in\mathfrak{S}_\infty/\mathrm{Stab}_\sigma$.
\item[\rm (c)] In particular, if $\mathcal{F}_\sigma,\mathcal{F}_\tau$ are both maximal (i.e., $\sigma,\tau$ are both bijective), then
\begin{eqnarray*}
\displaystyle \mathbf{X}(\mathcal{F}_\tau,E)\times\mathbf{X}(\mathcal{F}_\sigma,E)=\bigsqcup_{w\in\mathfrak{S}_\infty}(\pmb{\mathbb{O}}_{\tau,\sigma})_w 
\end{eqnarray*}
where
\begin{eqnarray*}
(\pmb{\mathbb{O}}_{\tau,\sigma})_w:=\{(g\mathcal{F}_\tau,g\mathcal{F}_{\sigma w}):g\in\mathbf{G}(E)\}
\end{eqnarray*}
is a decomposition of $\mathbf{X}(\mathcal{F}_\tau,E)\times\mathbf{X}(\mathcal{F}_\sigma,E)$ into $\mathbf{G}(E)$-orbits.
\end{itemize}
\end{proposition}

\begin{remark}
The orbit $(\pmb{\mathbb{O}}_{\tau,\sigma})_w$ 
of Proposition \ref{P2-3.2}\,{\rm (c)}
actually consists of all couples of generalized flags $(\mathcal{F}_\tau(\underline{v}),\mathcal{F}_{\sigma w}(\underline{v}))$
weakly compatible with the basis $\underline{v}=(v_1,v_2,\ldots)$.
\end{remark}

Assume $\mathbf{V}$ is endowed with a nondegenerate symmetric or symplectic form $\omega$ whose values on the basis $E$ are given by the matrix $\Omega$ in (\ref{omega}).

%

\begin{definition}
A generalized flag $\mathcal{F}$ is called {\em $\omega$-isotropic} if 
the map $F\mapsto F^\perp:=\{x\in\mathbf{V}:\omega(x,y)=0\ \forall y\in F\}$ 
is a well-defined involution of $\mathcal{F}$.
\end{definition}

\begin{proposition}[\cite{Dimitrov-Penkov}]
\label{P2.3-3}
Let $\mathcal{F}$ be an $\omega$-isotropic generalized flag weakly compatible with $E$.
The set $\mathbf{X}_\omega(\mathcal{F},E)$ of all
$\omega$-isotropic generalized flags which are $E$-commensurable with $\mathcal{F}$
is a $\mathbf{G}(E,\omega)$-homogeneous, closed ind-subvariety of $\mathbf{X}(\mathcal{F},E)$.
\end{proposition}

Finally, we emphasize that one of the main features of classical ind-groups is that their Borel ind-subgroups are not $\mathrm{Aut}(\mathbf{G})$-conjugate.  Here are three examples of maximal generalized flags in $\mathbf{V}$, compatible with the basis $E$ and such that their stablizers in $\mathbf{G}(E)$ are pairwise not $\mathrm{Aut}(\mathbf{G})$-conjugate.

\begin{example}
{\rm (a)} Let $\sigma_1:\mathbb{N}^*\to(\mathbb{N}^*,<)$, $\ell\mapsto\ell$.
The generalized flag $\mathcal{F}_{\sigma_1}$ is an ascending chain of subspaces
$\mathcal{F}_{\sigma_1}=\{0=F_0\subset F_1\subset F_2\subset\ldots\}$ isomorphic to $(\mathbb{N},<)$ as an ordered set. 
\\
{\rm (b)} Let $\sigma_2:\mathbb{N}^*\to\Big(\{\frac{1}{n}:n\in\mathbb{Z}^*\},<\Big)$, $\ell\mapsto \frac{(-1)^\ell}{\ell}$.
The generalized flag $\mathcal{F}_{\sigma_2}$ is a chain of the form
$\mathcal{F}_{\sigma_2}=\{0=F_0\subset F_1\subset\ldots\subset F_{-2}\subset F_{-1}=\mathbf{V}\}$
and is not isomorphic as ordered set to a subset of $(\mathbb{Z},<)$.
\\
{\rm (c)} Let $\sigma_3:\mathbb{N}^*\to(\mathbb{Q},<)$ be a bijection.
In this case no subspace $F\in\mathcal{F}_{\sigma_3}$ has both immediate successor or immediate predecessor.

\end{example}


\section{Parametrization of orbits in the finite-dimensional case}

\label{S3}

In Sections~\ref{S3.1}-\ref{S3.3}, we state explicit parametrizations of the $K$- and $G^0$-orbits in the finite-dimensional case.  All proofs are given in Section~\ref{S3.5}.

\subsection{Types A1 and A2}
\label{S3.1}
Let the notation be as in Subsection~\ref{S:A1A2}.
The space $V=V_n:=\langle e_1,\ldots,e_n\rangle_\mathbb{C}$ is endowed with the symmetric
or symplectic form $\omega(x,y)={}^tx\cdot \Omega\cdot y$ and the conjugation $\gamma(x)=\Omega\overline{x}$ which actually stand for the restrictions to $V$ of the maps $\omega,\gamma$ introduced in Section \ref{S2.1}.
This allows us to define two involutions of the flag variety $X$:
\[
\mathcal{F}=(F_0,\ldots,F_n)\mapsto\mathcal{F}^\perp:=(F_n^\perp,\ldots,F_0^\perp)
\quad\mbox{and}\quad
\mathcal{F}\mapsto\gamma(\mathcal{F}):=(\gamma(F_0),\ldots,\gamma(F_n))
\]
where $F^\perp\subset V$ stands for the subspace orthogonal to $F$ with respect to $\omega$.

Let $K=\{g\in\mathrm{GL}(V):\mbox{$g$ preserves $\omega$}\}$
and $G^0=\{g\in\mathrm{GL}(V):\gamma g=g\gamma\}$.

By $\mathfrak{I}_n\subset\mathfrak{S}_n$ we denote the
subset of involutions.
If $n=2m$ is even, we let $\mathfrak{I}'_n\subset\mathfrak{I}_n$
be the subset of involutions $w$ without fixed points.

\begin{definition}
\label{DA12}
Let $w\in\mathfrak{I}_n$. Set $\epsilon:=1$ in type A1 and
$\epsilon:=-1$ in type A2. A basis $(v_1,\ldots,v_n)$ of $V$ such that
\[\omega(v_k,v_\ell)=\left\{\begin{array}{ll} 1 & \mbox{if $w_k=\ell\geq k$} \\ \epsilon & \mbox{if $w_k=\ell< k$} \\ 0 & \mbox{if $w_k\not=\ell$} \end{array}\right.
\ \mbox{ for all $k,\ell\in\{1,\ldots,n\}$}\] is said to be {\em
$w$-dual}. A basis $(v_1,\ldots,v_n)$ of $V$ such that
\[
\gamma(v_k)=\left\{\begin{array}{ll} \epsilon v_{w_k} & \mbox{if $w_k\geq k$}
\\ v_{w_k} & \mbox{if $w_k< k$}
\end{array}\right.
\ \mbox{ for all $k\in\{1,\ldots,n\}$}
\]
is said to be {\em $w$-conjugate}. Set
\begin{eqnarray*}
 & \mathcal{O}_w=\{\mathcal{F}(v_1,\ldots,v_n):\mbox{$(v_1,\ldots,v_n)$ is a $w$-dual basis}\}, \\[1mm]
 & \mathfrak{O}_w=\{\mathcal{F}(v_1,\ldots,v_n):\mbox{$(v_1,\ldots,v_n)$ is a $w$-conjugate basis}\}.
\end{eqnarray*}
\end{definition}


\begin{proposition}
\label{P1} Let $\mathfrak{I}_n^\epsilon=\mathfrak{I}_n$ in type A1
and $\mathfrak{I}_n^\epsilon=\mathfrak{I}'_n$ in type A2. Recall the notation
$\mathbb{O}_w$ and $w_0$ introduced in Section \ref{S2.2}.
\begin{itemize}
\item[\rm (a)]
For every $w\in\mathfrak{I}_n^\epsilon$ we have
$\mathcal{O}_w\not=\emptyset$, $\mathfrak{O}_w\not=\emptyset$
and 
\[\mathcal{O}_w\cap\mathfrak{O}_w=\{\mathcal{F}(v_1,\ldots,v_n):\mbox{$(v_1,\ldots,v_n)$ is both $w$-dual and $w$-conjugate}\}\not=\emptyset.\]
\item[\rm (b)]
For every $w\in \mathfrak{I}_n^\epsilon$,
\[
\mathcal{O}_w=\{\mathcal{F}\in
X:(\mathcal{F}^\perp,\mathcal{F})\in\mathbb{O}_{w_0w}\}
\quad\mbox{and}\quad \mathfrak{O}_w=\{\mathcal{F}\in
X:(\gamma(\mathcal{F}),\mathcal{F})\in\mathbb{O}_{w}\}.
\]
\item[\rm (c)]
The subsets $\mathcal{O}_w$ ($w\in\mathfrak{I}_n^\epsilon$) are
exactly the $K$-orbits of $X$. The subsets $\mathfrak{O}_w$
($w\in\mathfrak{I}_n^\epsilon$) are exactly the $G^0$-orbits of $X$.
\item[\rm (d)]
The map $\mathcal{O}_w\mapsto\mathfrak{O}_w$ is Matsuki duality.
\end{itemize}
\end{proposition}

\subsection{Type A3} 
\label{S3.2}
Let the notation be as in Subsection~\ref{S:A3}: the space $V=V_n=\langle e_1,\ldots,e_n\rangle_\mathbb{C}$ is endowed with the hermitian form $\phi(x,y)={}^t\overline{x}\Phi y$ and a conjugation $\delta(x)=\Phi x$ where $\Phi$ is a diagonal matrix with entries $\epsilon_1,\ldots,\epsilon_n\in\{+1,-1\}$ (the left upper $n\times n$-corner of the matrix $\Phi$ of Section \ref{S2.1}).

Set $V_+=\langle e_k:\epsilon_k=1\rangle_\mathbb{C}$ and $V_-=\langle e_k:\epsilon_k=-1\rangle_\mathbb{C}$.  Then $V=V_+\oplus V_-$.
Let $K=\{g\in\mathrm{GL}(V):\delta g=g\delta\}=\mathrm{GL}(V_+)\times\mathrm{GL}(V_-)$ and $G^0=\{g\in \mathrm{GL}(V):\mbox{$g$ preserves $\phi$}\}$.

As in Section \ref{S3.1} we get two involutions of the flag variety $X$:
\[\mathcal{F}=(F_0,\ldots,F_n)\mapsto\delta(\mathcal{F}):=(\delta(F_0),\ldots,\delta(F_n))\quad\mbox{and}\quad\mathcal{F}\mapsto\mathcal{F}^\dag:=(F_n^\dag,\ldots,F_0^\dag)\]
where $F^\dag\subset V$ stands for the orthogonal of $F\subset V$ with respect to $\phi$. 
The hermitian form on the quotient $F/(F\cap F^\dag)$ induced by $\phi$ is
nondegenerate; we denote its signature by $\varsigma(\phi:F)$. Given $\mathcal{F}=(F_0,\ldots,F_n)\in X$, let
\[\varsigma(\phi:\mathcal{F}):=\big(\varsigma(\phi:F_\ell)\big)_{\ell=1}^n\in(\{0,\ldots,n\}^2)^n.\]
Then
\[\varsigma(\delta:\mathcal{F}):=\big((\dim F_\ell\cap V_+,\dim F_\ell\cap V_-)\big)_{\ell=1}^n\in(\{0,\ldots,n\}^2)^n\]
records the relative position of $\mathcal{F}$ with respect to the subspaces
$V_+$ and
$V_-$.

\medskip

\paragraph{\bf Combinatorial notation}
We call a {\em signed involution} a pair $(w,\varepsilon)$ consisting of
an involution $w\in\mathfrak{I}_n$ and signs
$\varepsilon_k\in\{+1,-1\}$ attached to its fixed points
$k\in\{\ell:w_\ell=\ell\}$.
(Equivalently $\varepsilon$ is a map $\{\ell:w_\ell=\ell\}\to\{+1,-1\}$.)

It is convenient to represent $w$ by a graph $l(w)$ (called {\em link pattern}) with $n$ vertices
$1,2,\ldots,n$ and an arc $(k,w_k)$ connecting $k$ and $w_k$ whenever $k<w_k$.
The {\em signed link pattern $l(w,\varepsilon)$} is
obtained from the graph $l(w)$ by marking each vertex
$k\in\{\ell:w_\ell=\ell\}$ with the label $+$ or $-$ depending on
whether $\varepsilon_k=+1$ or $\varepsilon_k=-1$.

For instance, the signed link pattern (where the numbering of vertices is implicit)
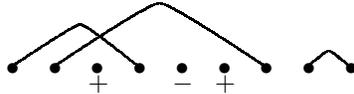
\[
\begin{picture}(140,30)(-20,-5)
\courbe{-16}{32}{0}{8}{10}{12}{16} \courbe{0}{80}{0}{36}{40}{44}{24}
\courbe{96}{112}{0}{103}{104}{105}{6} \put(-18,-3){$\bullet$}
\put(-2,-3){$\bullet$} \put(14,-3){$\bullet$} \put(13,-9){$+$}
\put(30,-3){$\bullet$} \put(46,-3){$\bullet$} \put(45,-9){$-$}
\put(62,-3){$\bullet$} \put(61,-9){$+$} \put(78,-3){$\bullet$}
\put(94,-3){$\bullet$} \put(110,-3){$\bullet$}
\end{picture}
\]
represents $(w,\varepsilon)$ with
$w=(1;4)(2;7)(8;9)\in\mathfrak{I}_9$ and
$(\varepsilon_3,\varepsilon_5,\varepsilon_6)=(+1,-1,+1)$.

We define $\varsigma(w,\varepsilon):=\{(p_\ell,q_\ell)\}_{\ell=1}^n$
as the sequence given by
\[
p_\ell\mbox{ (resp., $q_\ell$)}=\parbox[t]{8cm}{(number of $+$ signs
\mbox{ (resp., $-$ signs)} and arcs among the first $\ell$ vertices of
$l(w,\varepsilon)$).}
\]
Assuming $n=p+q$, let $\mathfrak{I}_n(p,q)$ be the set of
signed involutions of signature $(p,q)$, i.e., such that $(p_n,q_n)=(p,q)$.
Note that the elements of $\mathfrak{I}_n(p,q)$ coincide with the clans of signature $(p,q)$
in the sense of \cite{MO,Y}.

For instance, for the above pair
$(w,\varepsilon)$ we have
$(w,\varepsilon)\in\mathfrak{I}_9(5,4)$ and
\[
\varsigma(w,\varepsilon)=\big((0,0),(0,0),(1,0),(2,1),(2,2),(3,2),(4,3),(4,3),(5,4)\big).
\]

\begin{definition}
\label{D2}
\label{DA3}
Given a signed involution
$(w,\varepsilon)$, we say that
a basis $(v_1,\ldots,v_n)$ of $V$ is \emph{$(w,\epsilon)$-conjugate} if
\[
\delta(v_k)=\left\{\begin{array}{ll} \varepsilon_kv_{w_k} & \mbox{if
$w_k=k$}
\\ v_{w_k} & \mbox{if $w_k\not=k$}
\end{array}\right.
\ \mbox{ for all $k\in\{1,\ldots,n\}$\,.}
\]
A basis
$(v_1,\ldots,v_n)$ such that
\[\phi(v_k,v_\ell)=\left\{\begin{array}{ll} \varepsilon_k & \mbox{if $w_k=\ell=k$} \\ 1 & \mbox{if $w_k=\ell\not=k$} \\ 0 & \mbox{if $w_k\not=\ell$} \end{array}\right.
\ \mbox{ for all $k,\ell\in\{1,\ldots,n\}$}\] is said to be {\em
$(w,\varepsilon)$-dual}.
We set
\begin{eqnarray*}
 & \mathcal{O}_{(w,\varepsilon)}=\{\mathcal{F}(v_1,\ldots,v_n):\mbox{$(v_1,\ldots,v_n)$ is a $(w,\varepsilon)$-conjugate basis}\}, \\[1mm]
 & \mathfrak{O}_{(w,\varepsilon)}=\{\mathcal{F}(v_1,\ldots,v_n):\mbox{$(v_1,\ldots,v_n)$ is a $(w,\varepsilon)$-dual basis}\}.
\end{eqnarray*}
\end{definition}

\begin{proposition}
\label{P2} 
In addition to the above notation, let $(p,q)=(\dim V_+,\dim V_-)$. Then:
\begin{itemize}
\item[\rm (a)] For every $(w,\varepsilon)\in\mathfrak{I}_n(p,q)$ the subsets $\mathcal{O}_{(w,\varepsilon)}$ and $\mathfrak{O}_{(w,\varepsilon)}$ are nonempty, and 
\[\mathcal{O}_{(w,\varepsilon)}\cap\mathfrak{O}_{(w,\varepsilon)}
=\{\mathcal{F}(\underline{v}):\mbox{$\underline{v}=(v_k)_{k=1}^n$ is $(w,\varepsilon)$-dual and $(w,\varepsilon)$-conjugate}\}\not=\emptyset.\]
\item[\rm (b)]
For every $(w,\varepsilon)\in \mathfrak{I}_n(p,q)$, 
\begin{eqnarray*}
 & \mathcal{O}_{(w,\varepsilon)}=\big\{\mathcal{F}\in
X: (\delta(\mathcal{F}),\mathcal{F})\in\mathbb{O}_{w}\mbox{ and
}\varsigma(\delta:\mathcal{F})=\varsigma(w,\varepsilon)\big\}, \\
 &\mathfrak{O}_{(w,\varepsilon)}=\big\{\mathcal{F}\in X:
(\mathcal{F}^\dag,\mathcal{F})\in\mathbb{O}_{w_0w}\mbox{ and
}\varsigma(\phi:\mathcal{F})=\varsigma(w,\varepsilon)\big\}.
\end{eqnarray*}
\item[\rm (c)]
The subsets $\mathcal{O}_{(w,\varepsilon)}$
($(w,\varepsilon)\in\mathfrak{I}_n(p,q)$) are exactly the $K$-orbits
of $X$. The subsets $\mathfrak{O}_{(w,\varepsilon)}$
($(w,\varepsilon)\in\mathfrak{I}_n(p,q)$) are exactly the
$G^0$-orbits of $X$.
\item[\rm (d)]
The map
$\mathcal{O}_{(w,\varepsilon)}\mapsto\mathfrak{O}_{(w,\varepsilon)}$
is Matsuki duality.
\end{itemize}
\end{proposition}

\subsection{Types B, C, D}

\label{S3.3}

In this section we assume that the space $V=V_n=\langle e_1,\ldots,e_n\rangle_\mathbb{C}$ is endowed with a symmetric or symplectic form $\omega$ whose action on the basis $(e_1,\ldots,e_n)$ is described by the matrix $\Omega$
in (\ref{omega}).
We consider the group $G=G(V,\omega)=\{g\in\mathrm{GL}(V):\mbox{$g$ preserves $\omega$}\}$ and the variety of isotropic flags $X_\omega=\{\mathcal{F}\in X:\mathcal{F}^\perp=\mathcal{F}\}$
(see Section \ref{S2.2}).

In addition we assume that $V$ is endowed with a hermitian form $\phi$, a conjugation $\delta$, and a decomposition $V=V_+\oplus V_-$ (as in Section \ref{S3.2}) such that
\begin{itemize}
\item in types BD1 and C2, the restriction of $\omega$ to $V_+$ and $V_-$ is nondegenerate, i.e., $V_+^\perp=V_-$, 
\item in types C1 and D3, $V_+$ and $V_-$ are Lagrangian with respect to $\omega$, i.e., $V_+^\perp=V_+$ and $V_-^\perp=V_-$.
\end{itemize}
Set $K:=\{g\in G:g\delta=\delta g\}$ and $G^0:=\{g\in G:\mbox{$g$ preserves $\phi$}\}$.

\medskip

\paragraph{\bf Combinatorial notation}
Recall that $w_0(k)=n-k+1$. 
Let $(\eta,\epsilon)\in\{1,-1\}^2$.
A signed involution $(w,\varepsilon)$ is called {\em $(\eta,\epsilon)$-symmetric} if the following conditions hold
\begin{itemize}
\item[(i)] $ww_0=w_0w$ (so that the set $\{\ell:w_\ell=\ell\}$ is $w_0$-stable);
\item[(ii)] $\varepsilon_{w_0(k)}=\eta \varepsilon_k$ for all $k\in\{\ell:w_\ell=\ell\}$;
\end{itemize}
and in the case where $\eta\not=\epsilon$:
\begin{itemize}
\item[(iii)]$w_k\not=w_0(k)$ for all $k$.
\end{itemize}
Assuming $n=p+q$, let $\mathfrak{I}_n^{\eta,\epsilon}(p,q)\subset\mathfrak{I}_n(p,q)$ denote the subset of signed involutions of signature $(p,q)$ which are $(\eta,\epsilon)$-symmetric.

Specifically, $(w,\varepsilon)$ is $(1,1)$-symmetric when the signed link pattern $l(w,\varepsilon)$ is symmetric;
$(w,\varepsilon)$ is $(1,-1)$-symmetric when  $l(w,\varepsilon)$ is symmetric and does not have symmetric arcs (i.e., joining $k$ and $n-k+1$);
$(w,\varepsilon)$ is $(-1,-1)$-symmetric
when  $l(w,\varepsilon)$ is antisymmetric
in the sense that the mirror image of $l(w,\varepsilon)$ is a signed link pattern with the same arcs but opposite signs; $(w,\varepsilon)$ is $(-1,1)$-symmetric when $l(w,\varepsilon)$ is antisymmetric and does not have symmetric arcs. For instance:
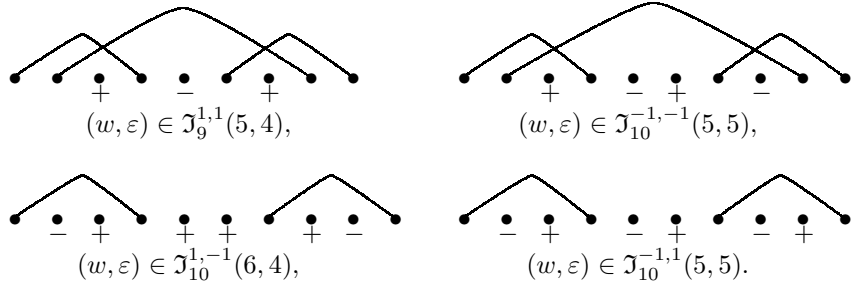
\begin{eqnarray*}
 & \begin{array}[t]{c}
\begin{picture}(140,30)(-20,-5)
\courbe{-16}{32}{0}{8}{10}{12}{16} \courbe{0}{96}{0}{40}{48}{56}{26}
\courbe{64}{112}{0}{86}{88}{90}{16} \put(-18,-3){$\bullet$}
\put(-2,-3){$\bullet$} \put(14,-3){$\bullet$} \put(13,-9){$+$}
\put(30,-3){$\bullet$} \put(46,-3){$\bullet$} \put(45,-9){$-$}
\put(62,-3){$\bullet$} \put(77,-9){$+$} \put(78,-3){$\bullet$}
\put(94,-3){$\bullet$} \put(110,-3){$\bullet$}
\end{picture}
\\[1mm]
\mbox{$(w,\varepsilon)\in\mathfrak{I}^{1,1}_9(5,4),$}
\end{array}
\qquad
\begin{array}[t]{c}
\begin{picture}(140,30)(-20,-5)
\courbe{-16}{32}{0}{8}{10}{12}{16} \courbe{0}{112}{0}{48}{56}{64}{28}
\courbe{80}{128}{0}{102}{104}{106}{16} \put(-18,-3){$\bullet$}
\put(-2,-3){$\bullet$} \put(14,-3){$\bullet$} \put(13,-9){$+$}
\put(30,-3){$\bullet$} \put(46,-3){$\bullet$} \put(45,-9){$-$}
\put(62,-3){$\bullet$} \put(61,-9){$+$} \put(78,-3){$\bullet$}
\put(94,-3){$\bullet$} \put(93,-9){$-$} \put(110,-3){$\bullet$} \put(126,-3){$\bullet$}
\end{picture}
\\[1mm]
\mbox{$(w,\varepsilon)\in\mathfrak{I}^{-1,-1}_{10}(5,5)$,}
\end{array} \\
 & \begin{array}[t]{c}
\begin{picture}(140,30)(-20,-5)
\courbe{-16}{32}{0}{8}{10}{12}{16}
\courbe{80}{128}{0}{102}{104}{106}{16} \put(-18,-3){$\bullet$}
\put(-2,-3){$\bullet$} \put(-3,-9){$-$} \put(14,-3){$\bullet$} \put(13,-9){$+$}
\put(30,-3){$\bullet$} \put(46,-3){$\bullet$} \put(45,-9){$+$}
\put(62,-3){$\bullet$} \put(61,-9){$+$} \put(78,-3){$\bullet$}
\put(94,-3){$\bullet$} \put(93,-9){$+$} \put(110,-3){$\bullet$} \put(109,-9){$-$} \put(126,-3){$\bullet$}
\end{picture}\\[1mm]
\mbox{$(w,\varepsilon)\in\mathfrak{I}^{1,-1}_{10}(6,4),$}
\end{array}
\qquad
\begin{array}[t]{c}
\begin{picture}(140,30)(-20,-5)
\courbe{-16}{32}{0}{8}{10}{12}{16}
\courbe{80}{128}{0}{102}{104}{106}{16} \put(-18,-3){$\bullet$}
\put(-2,-3){$\bullet$} \put(-3,-9){$-$} \put(14,-3){$\bullet$} \put(13,-9){$+$}
\put(30,-3){$\bullet$} \put(46,-3){$\bullet$} \put(45,-9){$-$}
\put(62,-3){$\bullet$} \put(61,-9){$+$} \put(78,-3){$\bullet$}
\put(94,-3){$\bullet$} \put(93,-9){$-$} \put(110,-3){$\bullet$} \put(109,-9){$+$} \put(126,-3){$\bullet$}
\end{picture}
\\[1mm]
\mbox{$(w,\varepsilon)\in\mathfrak{I}^{-1,1}_{10}(5,5)$.}
\end{array}
\end{eqnarray*}

\begin{proposition}
\label{P3}
Let $(p,q)=(\dim V_+,\dim V_-)$ (so that $p=q=\frac{n}{2}$ in types C1 and D3).
Set 
$(\eta,\epsilon)=(1,1)$ in type BD1,
$(\eta,\epsilon)=(1,-1)$ in type C2,
$(\eta,\epsilon)=(-1,-1)$ in types C1, and 
$(\eta,\epsilon)=(-1,1)$ in type D3.
\begin{itemize}
\item[(a)] For every $(w,\varepsilon)\in\mathfrak{I}_n^{\eta,\epsilon}(p,q)$,
considering bases $\underline{v}=(v_1,\ldots,v_n)$ of $V$ such that
\begin{equation}
\label{3.3.18}
\omega(v_k,v_\ell)=\left\{\begin{array}{ll} 0 & \mbox{if $\ell\not=n-k+1$} 
\\ 
1 & \mbox{if $\ell=n-k+1$ and $w_k,w_\ell\in[k,\,\ell]$ ($k\leq\ell$)} \\ 
\epsilon & \mbox{if $\ell=n-k+1$ and $w_k,w_\ell\in[\ell,\,k]$ ($\ell\leq k$)} \\
\eta & \mbox{if $\ell=n-k+1$ and $k,\ell\in]w_k,\,w_\ell[$} \\
\eta\epsilon & \mbox{if $\ell=n-k+1$ and $k,\ell\in]w_\ell,\,w_k[$,} \\
\end{array}\right.
\end{equation}
we have
\begin{eqnarray*}
 & & \mathcal{O}^{\eta,\epsilon}_{(w,\varepsilon)}:=\mathcal{O}_{(w,\varepsilon)}\cap X_\omega=\{\mathcal{F}(\underline{v}):\underline{v}\mbox{ is $(w,\varepsilon)$-conjugate and satisfies (\ref{3.3.18})}\}\not=\emptyset, \\
 & & \mathfrak{O}^{\eta,\epsilon}_{(w,\varepsilon)}:=\mathfrak{O}_{(w,\varepsilon)}\cap X_\omega=\{\mathcal{F}(\underline{v}):\underline{v}\mbox{ is $(w,\varepsilon)$-dual and satisfies (\ref{3.3.18})}\}\not=\emptyset, \\
 & & \mathcal{O}_{(w,\varepsilon)}^{\eta,\epsilon}\cap\mathfrak{O}_{(w,\varepsilon)}^{\eta,\epsilon}
 \\
 &&\phantom{aaa}=\{\mathcal{F}(\underline{v}):\underline{v}\mbox{ is $(w,\varepsilon)$-conjugate and $(w,\varepsilon)$-dual and satisfies (\ref{3.3.18})}\}\not=\emptyset.
\end{eqnarray*}
\item[(b)] The subsets $\mathcal{O}_{(w,\varepsilon)}^{\eta,\epsilon}$ ($(w,\varepsilon)\in\mathfrak{I}_n^{\eta,\epsilon}(p,q)$) are exactly the $K$-orbits of $X_\omega$. 
The subsets $\mathfrak{O}_{(w,\varepsilon)}^{\eta,\epsilon}$ ($(w,\varepsilon)\in\mathfrak{I}_n^{\eta,\epsilon}(p,q)$) are exactly the $G^0$-orbits of $X_\omega$.
\item[(c)] The map $\mathcal{O}_{(w,\varepsilon)}^{\eta,\epsilon}\mapsto \mathfrak{O}_{(w,\varepsilon)}^{\eta,\epsilon}$ is Matsuki duality.
\end{itemize}
\end{proposition}

\subsection{Remarks}

Set $X_0:=X$ in type A and $X_0:=X_\omega$ in types B, C, D.

\begin{remark}
\label{R3.4.1}
The characterization of the $K$-orbits in Propositions \ref{P1}--\ref{P3} can be stated in the following unified way. For $\mathcal{F}\in X$ we write
$\sigma(\mathcal{F})=\mathcal{F}^\perp$ in types A1--A2 and
$\sigma(\mathcal{F})=\delta(\mathcal{F})$ in types A3, BD1, C1--C2, D3. Let $P\subset
G$ be a parabolic subgroup containing $K$ and which is minimal for
this property. Two flags $\mathcal{F}_1,\mathcal{F}_2\in X_0$ belong
to the same $K$-orbit if and only if
$(\sigma(\mathcal{F}_1),\mathcal{F}_1)$ and
$(\sigma(\mathcal{F}_2),\mathcal{F}_2)$ belong to the same orbit of $P$ for the diagonal action of $P$ on $X_0\times X_0$.
\end{remark}

\begin{remark}[Open $K$-orbits]
\label{R-open}
With the notation of Remark \ref{R3.4.1} the map $\sigma_0:X_0\to X\times X$, $\mathcal{F}\mapsto(\sigma(\mathcal{F}),\mathcal{F})$ is a closed embedding.

In types A and C the flag variety $X_0$ is irreducible.
In particular there is a unique $G$-orbit $\mathbb{O}_w\subset X\times X$
such that $\mathbb{O}_w\cap \sigma_0(X_0)$ is open in $\sigma_0(X_0)$;
it corresponds to an element $w\in\mathfrak{S}_n$ maximal for the Bruhat order such that $\mathbb{O}_w$ intersects $\sigma_0(X_0)$.
In each case one finds a unique $K$-orbit $\mathcal{O}\subset X_0$ such that $\sigma_0(\mathcal{O})\subset\mathbb{O}_w$, it is therefore the (unique) open $K$-orbit of $X_0$.
This yields the following list of open $K$-orbits in types A1--A3, C1--C2:
\begin{itemize}
\item[\rm A1:] $\mathcal{O}_\mathrm{id}$;
\item[\rm A2:] $\mathcal{O}_{v_0}$ where $v_0=(1;2)(3;4)\cdots(n-1;n)$;
\item[\rm A3:] $\mathcal{O}_{(w_0^{(t)},\varepsilon)}$ where 
$t=\min\{p,q\}$, $\varepsilon\equiv\mathrm{sign}(p-q)$,
and $w^{(t)}_0=\prod\limits_{k=1}^t(k;n-k+1)$;
\item[\rm C1:] $\mathcal{O}^{-1,-1}_{(w_0,\emptyset)}$;
\item[\rm C2:] $\mathcal{O}^{1,-1}_{(\hat{w}_0^{(t)},\varepsilon)}$
where $t=\min\{p,q\}$, $\varepsilon\equiv\mathrm{sign}(p-q)$, and $\hat{w}_0^{(t)}=v_0^{(t)}w_0^{(t)}v_0^{(t)}$,
where $v_0^{(t)}=(1;2)(3;4)\cdots(t-1;t)$.
\end{itemize}

If $n=\dim V$ is even and the form $\omega$ is orthogonal, then the variety $X_\omega$ has two connected components.
In fact, for every isotropic flag $\mathcal{F}=(F_k)_{k=0}^n\in X_\omega$ there is a unique $\tilde{\mathcal{F}}=(\tilde{F}_k)_{k=0}^n\in X_\omega$
such that $F_k=\tilde{F}_k$ for all $k\not=m:=\frac{n}{2}$, $\tilde{F}_m\not=F_m$.
Then the map $\tilde{I}:\mathcal{F}\mapsto\tilde{\mathcal{F}}$ is an automorphism
of $X_\omega$ which maps one component of $X_\omega$ onto the other.
If $\mathcal{F}=\mathcal{F}(v_1,\ldots,v_n)$ for a basis $\underline{v}=(v_1,\ldots,v_n)$ such that 
\[\omega(v_k,v_\ell)\not=0\Leftrightarrow \ell=n-k+1\]
then
$\tilde{I}(\mathcal{F}(\underline{v}))=\mathcal{F}(\underline{\tilde{v}})$
where $\underline{\tilde{v}}$ is the basis obtained from $\underline{v}$ by switching the two middle vectors $v_m,v_{m+1}$.
If $\underline{v}$ is $(w,\varepsilon)$-conjugate then $\underline{\tilde{v}}$ is $\tilde{i}(w,\varepsilon)$-conjugate where
$\tilde{i}(w,\varepsilon):=\big((m;m+1)w(m;m+1),\varepsilon\circ(m;m+1)\big)$.
Hence $\tilde{I}$ maps the $K$-orbit $\mathcal{O}^{\eta,\epsilon}_{(w,\varepsilon)}$
onto $\mathcal{O}^{\eta,\epsilon}_{\tilde{i}(w,\varepsilon)}$.

In type D3, 
$X_\omega$ has exactly two open $K$-orbits.
More precisely $w=\hat{w}_0:=w_0v_0$
is maximal for the Bruhat order such that $\mathbb{O}_w\cap\sigma_0(X_0)$ is nonempty,
hence $\sigma_0^{-1}(\mathbb{O}_{\hat{w}_0})$ is open.
The permutation $\hat{w}_0$ has no fixed point if $m:=\frac{n}{2}$ is even; if $m:=\frac{n}{2}$ is odd, $\hat{w}_0$ fixes $m$ and $m+1$.
In the former case
$\sigma_0^{-1}(\mathbb{O}_{\hat{w}_0})=\mathcal{O}^{-1,1}_{(\hat{w}_0,\emptyset)}$
is a single $K$-orbit, and
$\tilde{I}(\mathcal{O}^{-1,1}_{(\hat{w}_0,\emptyset)})=\mathcal{O}^{-1,1}_{\tilde{i}(\hat{w}_0,\emptyset)}$
is a second open $K$-orbit.
In the latter case
$\sigma_0^{-1}(\mathbb{O}_{\hat{w}_0^{(m-1)}})=\mathcal{O}^{-1,1}_{(\hat{w}_0,\varepsilon)}\cup\mathcal{O}^{-1,1}_{(\hat{w}_0,\tilde\varepsilon)}$, where $(\varepsilon_m,\varepsilon_{m+1})=(\tilde\varepsilon_{m+1},\tilde\varepsilon_m)=(+1,-1)$, is the union of two distinct open $K$-orbits
which are image of each other by $\tilde{I}$.

In type BD1 the variety $X_\omega$ may be reducible
but $w=w^{(t)}_0$, for $t:=\min\{p,q\}$, is the unique maximal element of $\mathfrak{S}_n$ 
such that $\mathbb{O}_w\cap\sigma_0(X_0)$ is nonempty.
Then $\sigma_0^{-1}(\mathbb{O}_w)$
consists of a single $\tilde{I}$-stable open $K$-orbit, namely
$\mathcal{O}^{1,1}_{(w^{(t)}_0,\varepsilon)}$ for $\varepsilon\equiv\mathrm{sign}(p-q)$. 
The flag variety $X_\omega$ has therefore a unique open $K$-orbit (which is not connected whenever $n$ is even).
\end{remark}

\begin{remark}[Closed $K$-orbits]
\label{R-closed}
We use the notation of Remarks \ref{R3.4.1}--\ref{R-open}.
As seen from Propositions \ref{P1}--\ref{P3}, in each case
one finds a unique $w_{\mathrm{min}}\in\mathfrak{S}_n$ such that $\mathbb{O}_{w_{\mathrm{min}}}\cap\sigma_0(X_0)$ is closed; 
actually $w_{\mathrm{min}}=\mathrm{id}$ except in type BD1 for $p,q$ odd:
in that case $w_{\mathrm{min}}=(\frac{n}{2};\frac{n}{2}+1)$.
For every $K$-orbit $\mathcal{O}\subset X_0$ the following equivalence holds:
\[
\mbox{$\mathcal{O}$ is closed}\quad \Leftrightarrow\quad \sigma_0(\mathcal{O})\subset\mathbb{O}_{w_{\mathrm{min}}}
\]
(see \cite{Brion,Richardson-Springer}). In view of this equivalence, we deduce the following list of closed $K$-orbits of $X_0$ for the different types.
In types A1 and A2, $\mathcal{O}_{w_0}$ is the unique closed $K$-orbit.
In type A3 the closed $K$-orbits are exactly the orbits
$\mathcal{O}_{(\mathrm{id},\varepsilon)}$ for all pairs of the form $(\mathrm{id},\varepsilon)\in\mathfrak{I}_n(p,q)$;
there are $\binom{n}{p}$ such orbits.
In types B, C, D, the closed $K$-orbits are the orbits
$\mathcal{O}^{\eta,\epsilon}_{(\mathrm{id},\varepsilon)}$ for all pairs of the form $(\mathrm{id},\varepsilon)\in\mathfrak{I}^{\eta,\epsilon}_n(p,q)$,
except in type BD1 in the case where $n=:2m$ is even and $p,q$ are odd;
in that case the closed $K$-orbits
are the orbits $\mathcal{O}^{1,1}_{((m;m+1),\varepsilon)}$
for all pairs of the form $((m;m+1),\varepsilon)\in\mathfrak{I}^{1,1}_n(p,q)$.
There are $\binom{\lfloor\frac{p}{2}\rfloor+\lfloor\frac{q}{2}\rfloor}{\lfloor\frac{p}{2}\rfloor}$ closed orbits in types BD1 and C2, and there are $2^{\frac{n}{2}}$ closed orbits in types C1 and D3.
\end{remark}

\begin{remark} Propositions
\ref{P1}--\ref{P3} show in particular that the {\em special elements
of $X_0$}, in the sense of Matsuki \cite{Matsuki, Matsuki3}, are
precisely the flags $\mathcal{F}\in X_0$ of the form
$\mathcal{F}=\mathcal{F}(v_1,\ldots,v_n)$ where $(v_1,\ldots,v_n)$
is a basis of $V$ which is both dual and conjugate, with
respect to some involution $w\in\mathfrak{I}_n^\epsilon$ in types A1
and A2, and to some signed involution
$(w,\varepsilon)\in\mathfrak{I}_n(p,q)$ in types A3, B--D. Indeed, in view of \cite{Matsuki, Matsuki3} the set
$\mathcal{S}\subset X_0$ of special elements equals
\[\bigcup_{\mathcal{O}\in X_0/K}\mathcal{O}\cap\Xi(\mathcal{O})\]
where the map $X_0/K\to X_0/G^0$, $\mathcal{O}\mapsto\Xi(\mathcal{O})$
stands for Matsuki duality.
\end{remark}

\subsection{Proofs}

\label{S3.5}

\begin{proof}[Proof of Proposition \ref{P1}\,{\rm (a)}]
We write $w=(a_1;b_1)\cdots(a_m;b_m)$ with $a_1<\ldots<a_m$ and
$a_k<b_k$ for all $k$; let $c_1<\ldots<c_{n-2m}$ be the elements of
the set $\{k:w_k=k\}$.
In type A2 we have $n=2m$, and $(e_1,\ldots,e_n)$ is both a
$(1;2)(3;4)\cdots(n-1;n)$-dual basis and a
$(1;2)(3;4)\cdots(n-1;n)$-conjugate basis; then the basis
$\left\{e'_1,\ldots,e'_n\right\}$ given by
\[e'_{a_\ell}=e_{2\ell-1}\quad\mbox{and}\quad e'_{b_\ell}=e_{2\ell}\quad\mbox{for all $\ell\in\{1,\ldots,m\}$}\]
is simultaneously $w$-dual and $w$-conjugate.
In type A1, 
up to replacing $e_\ell$ and $e_{\ell^*}$ by $\frac{e_\ell+e_{\ell^*}}{\sqrt{2}}$ and $\frac{e_\ell-e_{\ell^*}}{i\sqrt{2}}$ whenever $\ell<\ell^*$,
we may assume that
the basis $(e_1,\ldots,e_n)$ is both $\mathrm{id}$-dual
and $\mathrm{id}$-conjugate. For every $\ell\in\{1,\ldots,m\}$ and
$k\in\{1,\ldots,n-2m\}$, we set
\[e'_{a_\ell}=\frac{e_{2\ell-1}+ie_{2\ell}}{\sqrt{2}},\quad e'_{b_\ell}=\frac{e_{2\ell-1}-ie_{2\ell}}{\sqrt{2}},\quad\mbox{and}\quad e'_{c_k}=e_{2m+k}.\]
Then $(e'_1,\ldots,e'_n)$ is simultaneously a $w$-dual and a
$w$-conjugate basis.
In both cases we conclude that
\begin{equation}
\label{3.1.4}
\emptyset\not=\{\mathcal{F}(v_1,\ldots,v_n):\mbox{$(v_1,\ldots,v_n)$ is $w$-dual and $w$-conjugate}\}\subset\mathcal{O}_w\cap\mathfrak{O}_w.
\end{equation}

Let us show the inverse inclusion. Assume
$\mathcal{F}=(F_0,\ldots,F_n)\in\mathcal{O}_w\cap\mathfrak{O}_w$.
Let $(v_1,\ldots,v_n)$ be a $w$-dual basis such that $\mathcal{F}=\mathcal{F}(v_1,\ldots,v_n)$.
Since $\mathcal{F}\in\mathfrak{O}_w$ we have
\begin{equation}
\label{1-new}
w_k=\min\{\ell=1,\ldots,n:\gamma(F_k)\cap F_\ell\not=\gamma(F_{k-1})\cap F_\ell\}.
\end{equation}
For all $\ell\in\{0,\ldots,n\}$ we will now construct a $w$-dual basis
$(v_1^{(\ell)},\ldots,v_n^{(\ell)})$ of $V$ such that
\begin{eqnarray}
\label{2-new}
& F_k=\langle v_1^{(\ell)},\ldots,v_k^{(\ell)}\rangle_\mathbb{C}
\quad\mbox{
for all $k\in\{1,\ldots,n\}$}
\end{eqnarray}
and
\begin{eqnarray}
\label{3-new}
&
\gamma(v_k^{(\ell)})=\left\{\begin{array}{ll} 
\epsilon v_{w_k}^{(\ell)} & \mbox{if $w_k\geq k$,} \\
v_{w_k}^{(\ell)} & \mbox{if $w_k<k$}
\end{array}\right.
\quad\mbox{
for all $k\in\{1,\ldots,\ell\}$.}
\end{eqnarray}
This will then imply
$\mathcal{F}=\mathcal{F}(v_1^{(n)},\ldots,v_n^{(n)})$ for a
basis $(v_1^{(n)},\ldots,v_n^{(n)})$ both $w$-dual and $w$-conjugate,
i.e., will complete the proof of (a).

Our construction is done by induction starting with $(v_1^{(0)},\ldots,v_n^{(0)})=(v_1,\ldots,v_n)$.
Let $\ell\in\{1,\ldots,n\}$, and assume that $(v_1^{(\ell-1)},\ldots,v_n^{(\ell-1)})$ is constructed.
We distinguish three cases.

\smallskip
\noindent
{\it Case 1: $w_\ell<\ell$.}

The inequality $w_\ell<\ell=w(w_\ell)$ implies $\gamma(v_{w_\ell}^{(\ell-1)})=\epsilon v_{\ell}^{(\ell-1)}$, whence
$\gamma(v_{\ell}^{(\ell-1)})=v_{w_\ell}^{(\ell-1)}$ as $\gamma^2=\epsilon\mathrm{id}$.
Therefore the basis $(v_1^{(\ell)},\ldots,v_n^{(\ell)}):=(v_1^{(\ell-1)},\ldots,v_n^{(\ell-1)})$ fulfills conditions (\ref{2-new}) and (\ref{3-new}).

\smallskip
\noindent
{\it Case 2: $w_\ell=\ell$.}

This case occurs only in type A1.
On the one hand, (\ref{1-new}) yields
\[\gamma(v_\ell^{(\ell-1)})\in\langle v_1^{(\ell-1)},\ldots,v_\ell^{(\ell-1)},v_{w_1}^{(\ell-1)},\ldots,v_{w_{\ell-1}}^{(\ell-1)}\rangle_\mathbb{C}.\]
On the other hand, since the basis $(v_1^{(\ell-1)},\ldots,v_n^{(\ell-1)})$ is $w$-dual,
we have
\[v_\ell^{(\ell-1)}\in\langle v_1^{(\ell-1)},\ldots,v_{\ell-1}^{(\ell-1)},v_{w_1}^{(\ell-1)},\ldots,v_{w_{\ell-1}}^{(\ell-1)}\rangle_\mathbb{C}^\perp\,.\]
Hence, as $\gamma$ preserves orthogonality with respect to $\omega$,
\begin{eqnarray*}
\gamma(v_\ell^{(\ell-1)}) & \in & \langle \gamma(v_1^{(\ell-1)}),\ldots,\gamma(v_{\ell-1}^{(\ell-1)}),\gamma(v_{w_1}^{(\ell-1)}),\ldots,\gamma(v_{w_{\ell-1}}^{(\ell-1)})\rangle_\mathbb{C}^\perp \\
 & & =\langle v_1^{(\ell-1)},\ldots,v_{\ell-1}^{(\ell-1)},v_{w_1}^{(\ell-1)},\ldots,v_{w_{\ell-1}}^{(\ell-1)}\rangle_\mathbb{C}^\perp.
\end{eqnarray*}
Altogether this yields a nonzero complex number $\lambda$ such that $\gamma(v_\ell^{(\ell-1)})=\lambda v_\ell^{(\ell-1)}$.
Since $\gamma$ is an involution, we have $\lambda\in\{+1,-1\}$.
In addition we know that 
\[\lambda=\omega(\gamma(v_\ell^{(\ell-1)}),v_\ell^{(\ell-1)})={}^t\overline{v_\ell^{(\ell-1)}}\cdot v_\ell^{(\ell-1)}\in\mathbb{R}^+.\]
Whence $\gamma(v_\ell^{(\ell-1)})=v_\ell^{(\ell-1)}$, and we can put
$(v_1^{(\ell)},\ldots,v_n^{(\ell)}):=(v_1^{(\ell-1)},\ldots,v_n^{(\ell-1)})$.

\smallskip
\noindent
{\it Case 3: $w_\ell>\ell$.}

By (\ref{1-new}) we have
\[\gamma(v_\ell^{(\ell-1)})\in
\langle v_k^{(\ell-1)}:1\leq k\leq w_\ell\rangle_\mathbb{C}+\langle v_{w_k}^{(\ell-1)}:1\leq k\leq \ell-1\rangle_\mathbb{C}.
\]
On the other hand, arguing as in Case 2 we see that
\[\gamma(v_\ell^{(\ell-1)})\in\langle v_1^{(\ell-1)},\ldots,v_{\ell-1}^{(\ell-1)},v_{w_1}^{(\ell-1)},\ldots,v_{w_{\ell-1}}^{(\ell-1)}\rangle_\mathbb{C}^\perp.\]
Hence we can write
\begin{equation}
\label{4-new}
\gamma(v_\ell^{(\ell-1)})=\sum_{k\in I}\lambda_kv_k^{(\ell-1)}
\quad\mbox{with $\lambda_k\in\mathbb{C}$ for all $k$,}
\end{equation}
where $I:=\{k:\ell\leq k\leq w_\ell\mbox{ and }\ell\leq w_k\}\subset\hat{I}:=\{k:\ell\leq k\mbox{ and }\ell\leq w_k\}$.
Using (\ref{4-new}), the fact that the basis $(v_1^{(\ell-1)},\ldots,v_n^{(\ell-1)})$ is $w$-dual, and the definition of $\omega$ and $\gamma$, we see that
\begin{equation}
\label{5-new}
\lambda_{w_\ell}=\omega(v_\ell^{(\ell-1)},\gamma(v_\ell^{(\ell-1)}))=\epsilon\cdot{}^tv_\ell^{(\ell-1)}\overline{v_\ell^{(\ell-1)}}=\epsilon\alpha
\end{equation}
with $\alpha\in\mathbb{R}$, $\alpha>0$.
Set 
\begin{eqnarray*}
 && v_\ell^{(\ell)}:=\frac{1}{\sqrt{\alpha}}v_\ell^{(\ell-1)},\quad
v_{w_\ell}^{(\ell)}:=\frac{\epsilon}{\sqrt{\alpha}}\gamma(v_\ell^{(\ell-1)}),\\
 && v_k^{(\ell)}:=v_k^{(\ell-1)}-\frac{\omega(v_k^{(\ell-1)},\gamma(v_\ell^{(\ell-1)}))}{\lambda_{w_\ell}}v_\ell^{(\ell-1)}\ \mbox{ for all $k\in \hat{I}\setminus\{\ell,w_\ell\}$}, \\
&& v_k^{(\ell)}:=v_k^{(\ell-1)}\ \mbox{ for all $k\in\{1,\ldots,n\}\setminus \hat{I}$.}
\end{eqnarray*}
Using (\ref{4-new}) and (\ref{5-new}) it is easy to check that $(v_1^{(\ell)},\ldots,v_n^{(\ell)})$ is a $w$-dual basis which satisfies (\ref{2-new}) and (\ref{3-new}).
This completes Case 3.
\end{proof}

\begin{proof}[Proof of Proposition \ref{P1}\,{\rm (b)}--{\rm (d)}]
Let $\mathcal{F}\in\mathcal{O}_w$, so
$\mathcal{F}=\mathcal{F}(v_1,\ldots,v_n)$ for some $w$-dual basis
$(v_1,\ldots,v_n)$ of $V$. From the definition of
$w$-dual basis we see that
\begin{eqnarray*}
\langle v_1,\ldots,v_{n-k}\rangle_\mathbb{C}^\perp & = & \langle v_j:w_j\notin\{1,\ldots,n-k\}\rangle_\mathbb{C} \\
 & = & \langle v_j:w_j\in\{n-k+1,\ldots,n\}\rangle_\mathbb{C} \\
 & = & \langle v_j:(w_0w)_j\in\{1,\ldots,k\}\rangle_\mathbb{C}\,.
\end{eqnarray*}
Therefore
\[\dim\langle v_1,\ldots,v_{n-k}\rangle_\mathbb{C}^\perp\cap\langle v_1,\ldots,v_\ell\rangle_\mathbb{C}=\big|\big\{j\in\{1,\ldots,\ell\}:(w_0w)_j\in\{1,\ldots,k\}\big\}\big|\]
for all $k,\ell\in\{1,\ldots,n\}$, which yields the equality
$w(\mathcal{F}^\perp,\mathcal{F})=w_0w$ and hence the inclusion
\begin{equation}
\label{1}
\mathcal{O}_w\subset\{\mathcal{F}\in
X:(\mathcal{F}^\perp,\mathcal{F})\in\mathbb{O}_{w_0w}\}.
\end{equation}
Let $\mathcal{F}=\mathcal{F}(v_1,\ldots,v_n)\in\mathfrak{O}_w$ for a
$w$-conjugate basis $(v_1,\ldots,v_n)$ of $V$. From the
definition of $w$-conjugate basis we get
\[\gamma(\langle v_1,\ldots,v_k\rangle_\mathbb{C})=\langle v_{w_j}:j\in\{1,\ldots,k\}\rangle_\mathbb{C}\,.\]
Therefore
\[\dim\gamma(\langle v_1,\ldots,v_k\rangle_\mathbb{C})\cap\langle v_1,\ldots,v_\ell\rangle_\mathbb{C}=\big|\big\{j\in\{1,\ldots,\ell\}:w^{-1}_j\in\{1,\ldots,k\}\big\}\big|\]
for all $k,\ell\in\{1,\ldots,n\}$, whence
$w(\gamma(\mathcal{F}),\mathcal{F})=w^{-1}=w$ (since $w$ is an
involution). This implies the inclusion
\begin{equation}
\label{2} \mathfrak{O}_w\subset\{\mathcal{F}\in
X:(\gamma(\mathcal{F}),\mathcal{F})\in\mathbb{O}_{w}\}.
\end{equation}

It is clear that the group $K$ acts transitively on the set of
$w$-dual bases, hence $\mathcal{O}_w$ is a $K$-orbit. Moreover
(\ref{1}) implies that the orbits $\mathcal{O}_w$ (for
$w\in\mathfrak{I}_w^\epsilon$) are pairwise distinct. Similarly the
subsets $\mathfrak{O}_w$ (for $w\in\mathfrak{I}_w^\epsilon$) are
pairwise distinct $G^0$-orbits.

We denote by $L_k$ the $k\times k$ matrix with $1$ on the
antidiagonal and $0$ elsewhere. Let
$\underline{v}=(v_1,\ldots,v_n)$ be a $w_0$-dual basis, in other
words,
\[
\left\{
\begin{array}{ll}
\omega(v_k,v_{n+1-k})=\left\{\begin{array}{ll}
1 & \mbox{if $k\leq\frac{n+1}{2}$} \\[1mm] \epsilon & \mbox{if $k>\frac{n+1}{2}$}
\end{array}\right. \\[4mm]
\omega(v_k,v_\ell)=0\quad\mbox{if $\ell\not=n+1-k$;}
\end{array}
\right.
\]
hence $L:=\left(\omega(v_k,v_\ell)\right)_{1\leq k,\ell\leq n}$ is
the following matrix
\[
L=L_n\quad \mbox{(type A1)}\quad\mbox{or}\quad L=\left(\begin{matrix} 0 & L_m \\ -L_m & 0 \end{matrix}\right)\quad\mbox{(type A2, $n=2m$)}.
\]
The flag $\mathcal{F}_0:=\mathcal{F}(v_1,\ldots,v_n)$ satisfies the
condition $\mathcal{F}_0^\perp=\mathcal{F}_0$. By
Richardson--Springer \cite{Richardson-Springer} every $K$-orbit
$\mathcal{O}\subset X$ contains an element of the form
$g\mathcal{F}_0$ with $g\in G$ such that
$h:=L{}^t[g]_{\underline{v}}L^{-1}[g]_{\underline{v}}\in N$ where
$[g]_{\underline{v}}$ denotes the matrix of $g$ in the basis
$\underline{v}$ and $N$ stands for the group of invertible $n\times
n$ matrices with exactly one nonzero coefficient in each row
and each column. Note that
$Lh={}^t[g]_{\underline{v}}L[g]_{\underline{v}}$ also belongs to $N$
(as $L$ does) and is symmetric in type A1 and antisymmetric in
type A2. Consequently, there are $w\in\mathfrak{I}_n$ and constants
$t_1,\ldots,t_n\in\mathbb{C}^*$ such that the matrix
$Lh=:\left(a_{k,\ell}\right)_{1\leq k,\ell\leq n}$ has the following
entries:
\[
a_{k,\ell}=0\ \mbox{ if $\ell\not=w_k$,}\qquad
a_{k,w_k}=\left\{\begin{array}{ll} t_k & \mbox{if $w_k\geq k$}
\\ \epsilon t_k & \mbox{if $w_k\leq k$.}
\end{array}\right.
\]
Since $\epsilon=-1$ in type A2, we must have $w_k\not=k$ for all $k$,
hence $w\in\mathfrak{I}'_n$. Therefore in both cases
$w\in\mathfrak{I}^\epsilon_n$. 
For each $k\in\{1,\ldots,n\}$, we
choose $s_k=s_{w_k}\in\mathbb{C}^*$ such that $s_k^{-2}=t_k$
(note that $t_{w_k}=t_k$).
Thus
\[g\mathcal{F}_0=\mathcal{F}(s_1gv_1,\ldots,s_ngv_n)\,,\]
and for all $k,\ell\in\{1,\ldots,n\}$ we have
\[
\omega(s_kgv_k,s_\ell g
v_\ell)=s_ks_\ell\omega(gv_k,gv_\ell)=s_ks_\ell
a_{k,\ell}=\left\{\begin{array}{ll} 1 & \mbox{if $\ell=w_k\geq k$}
\\ \epsilon & \mbox{if $\ell=w_k< k$} \\ 0 & \mbox{if
$\ell\not=w_k$}.
\end{array}\right.
\]
Whence $g\mathcal{F}_0\in\mathcal{O}_w$. This yields
$\mathcal{O}=\mathcal{O}_w$.

We have shown that the subsets $\mathcal{O}_w$ (for
$w\in\mathfrak{I}^\epsilon_w$) are precisely the $K$-orbits of $X$. In
particular, $X=\bigcup_{w\in\mathfrak{I}^\epsilon_w}\mathcal{O}_w$ so
that the inclusion (\ref{1}) is actually an equality. By Matsuki duality the
number of $G^0$-orbits of $X$ is the same as the number of
$K$-orbits, hence the subsets $\mathfrak{O}_w$ (for
$w\in\mathfrak{I}^\epsilon_w$) are exactly the $G^0$-orbits of $X$.
Thereby equality holds in (\ref{2}). Finally we have shown parts
{\rm (b)} and {\rm (c)} of the statement.

Part (a) implies that,  for every $w\in\mathfrak{I}_n^\epsilon$, the intersection $\mathcal{O}_w\cap\mathfrak{O}_w$ is nonempty and consists of a single $K\cap G^0$-orbit.
This shows that the orbit $\mathfrak{O}_w$ is the Matsuki dual of $\mathcal{O}_w$ (see \cite{Matsuki3}), and part {\rm (d)} of the statement is also proved.
\end{proof}

\begin{proof}[Proof of Proposition \ref{P2}\,{\rm (a)}]
We write $w$ as a product of pairwise disjoint transpositions
$w=(a_1;b_1)\cdots(a_m;b_m)$,
and let $c_{m+1}<\ldots<c_p$ be the elements of $\{k:w_k=k,\ \varepsilon_k=+1\}$
and $d_{m+1}<\ldots<d_q$ be the elements of $\{k:w_k=k,\ \varepsilon_k=-1\}$.
Let $\{e_1,\ldots,e_n\}=\{e_1^+,\ldots,e_p^+\}\cup\{e_1^-,\ldots,e_q^-\}$
so that $V_+=\langle e_\ell^+:\ell=1,\ldots,p\rangle_\mathbb{C}$
and $V_-=\langle e_\ell^-:\ell=1,\ldots,q\rangle_\mathbb{C}$.
Setting
\begin{eqnarray*}
& v_{a_k}:=\frac{e^+_k+e^-_k}{\sqrt{2}}\,,\,\,\ v_{b_k}:=\frac{e^+_k-e^-_k}{\sqrt{2}}\ \mbox{ for all $k\in\{1,\ldots,m\}$,} \\
& v_{c_k}:=e^+_k\ \mbox{ for all $k\in\{m+1,\ldots,p\}$, \ and }\ v_{d_k}:=e^-_k\ \mbox{ for all $k\in\{m+1,\ldots,q\}$,}
\end{eqnarray*}
it is easy to see that $(v_1,\ldots,v_n)$ is a basis of $V$ which is
$(w,\varepsilon)$-dual and $(w,\varepsilon)$-conjugate.  Therefore
\begin{equation}
\label{3.5.14}
\emptyset\not=\{\mathcal{F}(\underline{v}):\mbox{$\underline{v}$ is $(w,\varepsilon)$-dual and $(w,\varepsilon)$-conjugate}\}\subset \mathcal{O}_{(w,\varepsilon)}\cap\mathfrak{O}_{(w,\varepsilon)}.
\end{equation}

	For showing the inverse inclusion, consider
$\mathcal{F}=(F_0,\ldots,F_n)\in \mathcal{O}_{(w,\varepsilon)}\cap\mathfrak{O}_{(w,\varepsilon)}$.
On the one hand, since $\mathcal{F}\in\mathfrak{O}_{(w,\varepsilon)}$ there is a $(w,\varepsilon)$-dual basis $(v_1,\ldots,v_n)$ such that
$\mathcal{F}=\mathcal{F}(v_1,\ldots,v_n)$. On the other hand, the fact that $\mathcal{F}\in\mathcal{O}_{(w,\varepsilon)}$ yields
\begin{equation}
\label{8-new}
w_k=\min\{\ell=1,\ldots,n:\delta(F_k)\cap F_\ell\not=\delta(F_{k-1})\cap F_\ell\}\ \mbox{ for all $k\in\{1,\ldots,n\}$.}
\end{equation}

For all $\ell\in\{0,\ldots,n\}$ we will now construct a $(w,\varepsilon)$-dual basis
$(v_1^{(\ell)},\ldots,v_n^{(\ell)})$ such that
\begin{eqnarray}
\label{9-new}
 & F_k=\langle v_1^{(\ell)},\ldots,v_k^{(\ell)}\rangle_\mathbb{C}\quad\mbox{for all $k\in\{1,\ldots,n\}$}
\\
\label{10-new} 
 \mbox{and} & \delta(v_k^{(\ell)})=\left\{\begin{array}{ll}
v_{w_k}^{(\ell)} & \mbox{if $w_k\not=k$,} \\
\varepsilon_k v_{k}^{(\ell)} & \mbox{if $w_k=k$} \\
\end{array}\right.\quad\mbox{for all $k\in\{1,\ldots,\ell\}$.}
\end{eqnarray}
This will then provide
a basis $(v_1^{(n)},\ldots,v_n^{(n)})$ which is both $(w,\varepsilon)$-dual and $(w,\varepsilon)$-conjugate and such that
$\mathcal{F}=\mathcal{F}(v_1^{(n)},\ldots,v_n^{(n)})$, i.e., will complete the proof of part {\rm (a)}.

The construction is carried out by induction on $\ell\in\{0,\ldots,n\}$, and
 is initialized by setting
$(v_1^{(0)},\ldots,v_n^{(0)}):=(v_1,\ldots,v_n)$. Let
$\ell\in\{1,\ldots,n\}$ be such that the basis
$(v_1^{(\ell-1)},\ldots,v_n^{(\ell-1)})$ is already constructed. We
distinguish three cases.

\medskip
\noindent {\it Case 1: $w_\ell<\ell$.}

Since in this case since $w_\ell\leq\ell-1$ and $w(w_\ell)=\ell$, we get
$\delta(v_{w_\ell}^{(\ell-1)})=v_{\ell}^{(\ell-1)}$ and hence
$\delta(v_{\ell}^{(\ell-1)})=v_{w_\ell}^{(\ell-1)}$ (as
$\delta$ is an involution). Therefore the basis
$(v_1^{(\ell)},\ldots,v_n^{(\ell)}):=(v_1^{(\ell-1)},\ldots,v_n^{(\ell-1)})$
satisfies conditions (\ref{9-new}) and (\ref{10-new}).

\medskip
\noindent {\it Case 2: $w_\ell=\ell$.}

Using (\ref{8-new}) we have
\[\delta(v_\ell^{(\ell-1)})\in\langle v_1^{(\ell-1)},v_2^{(\ell-1)},\ldots,v_\ell^{(\ell-1)}\rangle_\mathbb{C}+\langle v_{w_1}^{(\ell-1)},\ldots,v_{w_{\ell-1}}^{(\ell-1)}\rangle_\mathbb{C}\,.\]
On the other hand, the fact that the basis
$(v_1^{(\ell-1)},\ldots,v_n^{(\ell-1)})$ is
$(w,\varepsilon)$-conjugate implies
\begin{align}v_\ell^{(\ell-1)}\in\langle v_1^{(\ell-1)},\ldots,v_{\ell-1}^{(\ell-1)},v_{w_1}^{(\ell-1)},\ldots,v_{w_{\ell-1}}^{(\ell-1)}\rangle_\mathbb{C}^\dag.
\label{E:induction}
\end{align}
Since $\delta$ preserves orthogonality with respect to the form
$\phi$ and since $\delta(v_k^{(\ell-1)})=v_{w_k}^{(\ell-1)}$
for all $k\in\{1,\ldots,\ell-1\}$ (by the induction hypothesis), (\ref{E:induction}) yields
\[\delta(v_\ell^{(\ell-1)})\in\langle v_1^{(\ell-1)},\ldots,v_{\ell-1}^{(\ell-1)},v_{w_1}^{(\ell-1)},\ldots,v_{w_{\ell-1}}^{(\ell-1)}\rangle_\mathbb{C}^\dag.\]
Altogether we deduce that
\[\delta(v_\ell^{(\ell-1)})=\lambda
v_\ell^{(\ell-1)}\quad\mbox{for some $\lambda\in\mathbb{C}^*$.}\] 
As $\delta$ is an involution, we conclude that $\lambda\in\{+1,-1\}$.
Moreover, knowing that
$\phi(v_\ell^{(\ell-1)},v_\ell^{(\ell-1)})=\varepsilon_\ell$ we see
that
\[\lambda\varepsilon_\ell=\phi(v_\ell^{(\ell-1)},\delta(v_\ell^{(\ell-1)}))={}^t\overline{v_\ell^{(\ell-1)}}\Phi\Phi v_\ell^{(\ell-1)}={}^t\overline{v_\ell^{(\ell-1)}}v_\ell^{(\ell-1)}\geq 0.\]
Finally we conclude that $\lambda=\varepsilon_\ell$. It follows that
the basis
$(v_1^{(\ell)},\ldots,v_n^{(\ell)}):=(v_1^{(\ell-1)},\ldots,v_n^{(\ell-1)})$
satisfies (\ref{9-new}) and (\ref{10-new}).

\medskip
\noindent {\it Case 3: $w_\ell>\ell$.}

Invoking (\ref{8-new}), the fact that
$(v_1^{(\ell-1)},\ldots,v_n^{(\ell-1)})$ is
$(w,\varepsilon)$-dual, the induction hypothesis, and the fact
that $\delta$ preserves orthogonality with respect to $\phi$, we see as in Case 2 that
\begin{eqnarray*}
\delta(v_\ell^{(\ell-1)}) & \in & \big(\langle v_k^{(\ell-1)}:1\leq
k\leq w_\ell\rangle_\mathbb{C}+\langle v_{w_k}^{(\ell-1)}:1\leq k\leq
\ell-1\rangle_\mathbb{C}\big) \\
 & & \cap\,\langle
v_1^{(\ell-1)},\ldots,v_{\ell-1}^{(\ell-1)},v_{w_1}^{(\ell-1)},\ldots,v_{w_{\ell-1}}^{(\ell-1)}\rangle_\mathbb{C}^\dag.
\end{eqnarray*} Therefore
\begin{equation} \label{11-new}
\delta(v_\ell^{(\ell-1)})=\sum_{k\in
I}\lambda_kv_k^{(\ell-1)}\quad\mbox{with $\lambda_k\in\mathbb{C}$,}
\end{equation}
where $I:=\{k:\ell\leq k\leq w_\ell,\ \ell\leq w_k\}\subset\hat{I}:=\{k:\ell\leq k,\ \ell\leq w_k\}$.
This implies
\[\lambda_{w_\ell}=\phi(v_\ell^{(\ell-1)},\delta(v_\ell^{(\ell-1)}))={}^t\overline{v_\ell^{(\ell-1)}}\Phi\Phi v_\ell^{(\ell-1)}={}^t\overline{v_\ell^{(\ell-1)}}v_\ell^{(\ell-1)}\in\mathbb{R}_+^*.\]
It is straightforward to check that the basis
$(v_1^{(\ell)},\ldots,v_n^{(\ell)})$ defined by
\begin{eqnarray*}
 & & v_\ell^{(\ell)}:=\frac{1}{\sqrt{\lambda_{w_\ell}}}v_\ell^{(\ell-1)},\quad
v_{w_\ell}^{(\ell)}:=\frac{1}{\sqrt{\lambda_{w_\ell}}}\delta(v_\ell^{(\ell-1)}),
\\
 & & v_k^{(\ell)}:=v_k^{(\ell-1)}-\frac{\phi(v_k^{(\ell-1)},\delta(v_\ell^{(\ell-1)}))}{\lambda_{w_\ell}}v_\ell^{(\ell-1)}
\quad\mbox{for all $k\in \hat{I}\setminus\{\ell,w_\ell\}$,} \\
 &  & v_k^{(\ell)}:=v_k^{(\ell-1)}\quad\mbox{for all $k\in\{1,\ldots,n\}\setminus \hat{I}$}
\end{eqnarray*}
is $(w,\varepsilon)$-dual and satisfies conditions
(\ref{9-new}) and (\ref{10-new}).
\end{proof}

\begin{proof}[Proof of Proposition \ref{P2}\,{\rm (b)}--{\rm (d)}]
Let $\mathcal{F}=\mathcal{F}(v_1,\ldots,v_n)$ where
$(v_1,\ldots,v_n)$ is a $(w,\varepsilon)$-conjugate basis. Then by
definition we have
\[\delta(\langle v_1,\ldots,v_k\rangle_\mathbb{C})=\langle v_{w_j}:j\in\{1,\ldots,k\}\rangle_\mathbb{C}\,,\]
hence \begin{eqnarray*} \dim \delta(\langle
v_1,\ldots,v_k\rangle_\mathbb{C})\cap \langle v_1,\ldots,v_\ell\rangle_\mathbb{C} & = &
|\{j\in\{1,\ldots,\ell\}:w^{-1}_j\in\{1,\ldots,k\}\}| \\
  & = &
|\{j\in\{1,\ldots,\ell\}:w_j\in\{1,\ldots,k\}\}|
\end{eqnarray*}
for all $k,\ell\in\{1,\ldots,n\}$. Moreover, for
$\varepsilon\in\{+1,-1\}$ we have
\begin{eqnarray*}
\langle
v_1,\ldots,v_\ell\rangle_\mathbb{C}\cap\ker(\delta-\varepsilon\mathrm{id}) & =
&
\langle v_j:1\leq w_j=j\leq \ell\mbox{ and }\varepsilon_j=\varepsilon\rangle_\mathbb{C} \\
 & & +\langle
v_j+\varepsilon v_{w_j}:1\leq w_j<j\leq \ell\rangle_\mathbb{C}\,.
\end{eqnarray*}
Therefore
\[
\big(\dim \langle
v_1,\ldots,v_\ell\rangle_\mathbb{C}\cap V_+,\dim \langle
v_1,\ldots,v_\ell\rangle_\mathbb{C}\cap V_-\big)_{\ell=1}^n=\varsigma(w,\varepsilon).
\]
Altogether this yields the inclusion
\begin{equation}
\label{3} \mathcal{O}_{(w,\varepsilon)}\subset\big\{\mathcal{F}\in
X: (\delta(\mathcal{F}),\mathcal{F})\in\mathbb{O}_{w}\mbox{ and
}\varsigma(\delta:\mathcal{F})=\varsigma(w,\varepsilon)\big\}.
\end{equation}

Now let $(v_1,\ldots,v_n)$ be a $(w,\varepsilon)$-dual basis. Then
\begin{eqnarray*}
\langle v_1,\ldots,v_{n-k}\rangle_\mathbb{C}^\dag\cap\langle
v_1,\ldots,v_\ell\rangle_\mathbb{C} & = & \langle
v_j:j\in\{1,\ldots,\ell\}\mbox{ and }w_j>n-k\rangle_\mathbb{C} \\
 & = & \langle
v_j:j\in\{1,\ldots,\ell\}\mbox{ and }(w_0w)_j\leq k\rangle_\mathbb{C}\,,
\end{eqnarray*}
whence
\[\dim\langle v_1,\ldots,v_{n-k}\rangle_\mathbb{C}^\dag\cap\langle v_1,\ldots,v_\ell\rangle_\mathbb{C}=|\{j\in\{1,\ldots,\ell\}:(w_0w)_j\in\{1,\ldots,k\}|\]
for all $k,\ell\in\{1,\ldots,n\}$. In particular we see that
\[\langle v_1,\ldots,v_\ell\rangle_\mathbb{C}=\langle v_1,\ldots,v_\ell\rangle_\mathbb{C}\cap\langle v_1,\ldots,v_\ell\rangle_\mathbb{C}^\dag\oplus\langle v_j:
j\in\{1,\ldots,\ell\}\mbox{ and }w_j\leq \ell\rangle_\mathbb{C}.\] It follows
that the vectors $v_j$ (for $1\leq w_j=j\leq \ell$) and
$\frac{1}{\sqrt{2}}(v_j\pm v_{w_j})$ (for $1\leq w_j<j\leq \ell$)
form a basis of the quotient space $\langle
v_1,\ldots,v_\ell\rangle_\mathbb{C}/\langle v_1,\ldots,v_\ell\rangle_\mathbb{C}\cap\langle
v_1,\ldots,v_\ell\rangle_\mathbb{C}^\dag$. This basis is $\phi$-orthogonal and,
since $(v_1,\ldots,v_n)$ is $(w,\varepsilon)$-dual, we have
\[
\textstyle \phi(v_j,v_j)=\varepsilon_j\mbox{ if $w_j=j$};\quad
\left\{\begin{array}{l}
\phi\big(\frac{v_j+v_{w_j}}{\sqrt{2}},\frac{v_j+v_{w_j}}{\sqrt{2}}\big)=1,\\[2mm]
\phi\big(\frac{v_j-v_{w_j}}{\sqrt{2}},\frac{v_j-v_{w_j}}{\sqrt{2}}\big)=-1
\end{array}\right.\mbox{ if $w_j<j$.}
\]
Therefore the signature of $\phi$ on $\langle
v_1,\ldots,v_\ell\rangle_\mathbb{C}/\langle v_1,\ldots,v_\ell\rangle_\mathbb{C}\cap\langle
v_1,\ldots,v_\ell\rangle_\mathbb{C}^\dag$ is the pair
\[
\begin{array}{rll}
\big( & |\{j:w_j=j\leq\ell,\
\varepsilon_j=+1\}|+|\{j:w_j<j\leq\ell\}|,
\\[2mm]
 & |\{j:w_j=j\leq\ell,\
\varepsilon_j=-1\}|+|\{j:w_j<j\leq\ell\}| & \big)\end{array}
\]
which coincides with the $\ell$-th term of the sequence
$\varsigma(w,\varepsilon)$. Finally, we obtain the inclusion
\begin{equation}
\label{4} \mathfrak{O}_{(w,\varepsilon)}\subset\big\{\mathcal{F}\in
X: (\mathcal{F}^\dag,\mathcal{F})\in\mathbb{O}_{w_0w}\mbox{ and
}\varsigma(\phi:\mathcal{F})=\varsigma(w,\varepsilon)\big\}.
\end{equation}

It is clear that $K$ (resp., $G^0$) acts transitively on the set of
$(w,\varepsilon)$-conjugate bases (resp., $(w,\varepsilon)$-dual
bases). Hence the subsets $\mathcal{O}_{(w,\varepsilon)}$ (resp.
$\mathfrak{O}_{(w,\varepsilon)}$) are $K$-orbits (resp.,
$G^0$-orbits). Moreover, in view of (\ref{3}) and (\ref{4}) these
orbits are pairwise distinct.

Let $\mathcal{O}$ be a $K$-orbit of $X$. Note that the basis $(e_1,\ldots,e_n)$ of
$V$ satisfies $\delta(e_j)=\pm
e_j$ for all $j$, hence the flag
$\mathcal{F}_0:=\mathcal{F}(e_1,\ldots,e_n)$ satisfies
$\delta(\mathcal{F}_0)=\mathcal{F}_0$. By
\cite{Richardson-Springer} the $K$-orbit $\mathcal{O}$ contains an
element of the form $g\mathcal{F}_0$ for some $g\in G$ such that
$h:=\Phi g^{-1}\Phi g\in N$ where, as in the proof of Proposition
\ref{P1}, $N\subset G$ stands for the subgroup of matrices with
exactly one nonzero entry in each row and each column. Since
$\Phi\in N$ we also have $\Phi h\in N$. Hence there is a permutation
$w\in\mathfrak{S}_n$ and constants
$t_1,\ldots,t_n\in\mathbb{C}^*$ such that the matrix $\Phi
h=:\big(a_{k,\ell}\big)_{1\leq k,\ell\leq n}$ has entries
\[a_{k,\ell}=0\ \mbox{ if $\ell\not=w_k$,}\quad a_{k,w_k}=t_k\quad\mbox{ for all $k,\ell\in\{1,\ldots,n\}$}.\]
The relation $\Phi h=g^{-1}\Phi g$ shows that $(\Phi h)^2=1_n$. This
yields $w^2=\mathrm{id}$ and $t_kt_{w_k}=1$ for all $k$; hence
\[t_{w_k}=t_k^{-1}\ \mbox{ whenever $w_k\not=k$}\quad\mbox{and}\quad \varepsilon_k:=t_k\in\{+1,-1\}\ \mbox{ whenever
$w_k=k$}.\] In addition, since $\Phi h$ is conjugate to $\Phi$, its
eigenvalues $+1$ and $-1$ have respective multiplicities $p$ and
$q$, which forces
\[(w,\varepsilon)\in\mathfrak{I}_n(p,q).\]
For each $k\in\{1,\ldots,n\}$ with $w_k<k$, we
take $s_k\in\mathbb{C}^*$ such that $t_k=s_k^2$ and set
$s_{w_k}=s_k^{-1}$ (so that $s_{w_k}^2=t_k^{-1}=t_{w_k}$). Moreove,r
for each $k\in\{1,\ldots,n\}$ with $w_k=k$ we set $s_k=1$. The
equality $\Phi g=g\Phi h$ yields
\[
\delta(g(s_ke_k))=s_k\Phi g e_k=s_k g(\Phi h)e_k=s_k
g(t_{w_k}e_{w_k})=s_{w_k}^{-1}
g(s_{w_k}^2e_{w_k})=g(s_{w_k}e_{w_k})
\]
for all $k\in\{1,\ldots,n\}$ such that $w_k\not=k$, and
\[
\delta(g(s_k e_k))=\delta(g(e_k))=\Phi g e_k=g(\Phi
h)e_k=g(\varepsilon_k e_k)=\varepsilon_k g(e_k)=\varepsilon_k g(s_k
e_k)
\]
for all $k\in\{1,\ldots,n\}$ such that $w_k=k$. Hence the family
$(g(s_1 e_1),\ldots,g(s_n e_n))$ is a $(w,\varepsilon)$-conjugate basis of
$V$. Thus
\[g\mathcal{F}_0=g\mathcal{F}(e_1,\ldots,e_n)=g\mathcal{F}(s_1 e_1,\ldots,s_n e_n)=\mathcal{F}(g(s_1 e_1),\ldots,g(s_n e_n))\in\mathcal{O}_{(w,\varepsilon)}.\]
Therefore $\mathcal{O}=\mathcal{O}_{(w,\varepsilon)}$.

We conclude that the subsets $\mathcal{O}_{(w,\varepsilon)}$ (for
$(w,\varepsilon)\in\mathfrak{I}_n(p,q)$) are exactly the $K$-orbits
of $X$. Matsuki duality then guarantees that the subsets
$\mathfrak{O}_{(w,\varepsilon)}$ (for
$(w,\varepsilon)\in\mathfrak{I}_n(p,q)$) are exactly the
$G^0$-orbits of $X$. This fact implies in particular that
equality holds in (\ref{3}) and (\ref{4}). Altogether we have shown
parts {\rm (b)} and {\rm (c)} of the statement.

Finally, part {\rm (a)} shows that for every
$(w,\varepsilon)\in\mathfrak{I}_n(p,q)$ the intersection
$\mathcal{O}_{(w,\varepsilon)}\cap\mathfrak{O}_{(w,\varepsilon)}$
consists of a single $K\cap G^0$-orbit, which guarantees that the
orbits $\mathcal{O}_{(w,\varepsilon)}$ and
$\mathfrak{O}_{(w,\varepsilon)}$ are Matsuki dual (see
\cite{Matsuki,Matsuki3}). This proves part {\rm (d)} of the
statement. The proof of Proposition \ref{P2} is complete.
\end{proof}

\begin{proof}[Proof of Proposition \ref{P3}]
The proof relies on the following two technical claims.

\medskip
\item
{\it Claim 1:}
For every signed involution $(w,\varepsilon)\in\mathfrak{I}_n(p,q)$ we have
$\mathcal{O}_{(w,\varepsilon)}\cap X_\omega=\emptyset$ unless $(w,\varepsilon)\in\mathfrak{I}_n^{\eta,\epsilon}(p,q)$.

\medskip
\noindent
{\it Claim 2:}
For every $(w,\varepsilon)\in\mathfrak{I}_n^{\eta,\epsilon}(p,q)$ there is a basis $\underline{v}=(v_1,\ldots,v_n)$ which is simultaneously $(w,\varepsilon)$-dual and $(w,\varepsilon)$-conjugate and satisfies (\ref{3.3.18}).

\medskip

Assuming Claims 1 and 2, the proof of the proposition proceeds as follows. For every $(w,\varepsilon)\in\mathfrak{I}_n(p,q)$ the inclusions
\begin{eqnarray}
\label{proof.P3.1} & \{\mathcal{F}(\underline{v}):\mbox{$\underline{v}$ is $(w,\varepsilon)$-conjugate
 and satisfies (\ref{3.3.18})}\}\subset\mathcal{O}_{(w,\varepsilon)}\cap X_\omega,
 \\
\label{proof.P3.2} & \{\mathcal{F}(\underline{v}):\mbox{$\underline{v}$ is $(w,\varepsilon)$-dual
 and satisfies (\ref{3.3.18})}\}\subset\mathfrak{O}_{(w,\varepsilon)}\cap X_\omega, \\
\label{proof.P3.3} & \{\mathcal{F}(\underline{v}):\mbox{$\underline{v}$ is $(w,\varepsilon)$-dual and $(w,\varepsilon)$-conjugate and satisfies (\ref{3.3.18})}\}\\ \nonumber
& \subset
\mathcal{O}_{(w,\varepsilon)}\cap\mathfrak{O}_{(w,\varepsilon)}\cap X_\omega
\end{eqnarray}
clearly hold. Hence Claim 2 shows that 
$\mathcal{O}_{(w,\varepsilon)}^{\eta,\epsilon}$, $\mathfrak{O}_{(w,\varepsilon)}^{\eta,\epsilon}$, and
$\mathcal{O}_{(w,\varepsilon)}^{\eta,\epsilon}\cap\mathfrak{O}_{(w,\varepsilon)}^{\eta,\epsilon}$ are all nonempty whenever $(w,\varepsilon)\in\mathfrak{I}_n^{\eta,\epsilon}(p,q)$.
By Claim 1, Lemma \ref{lemma-2.2.1}, and Proposition \ref{P2}\,{\rm (c)}, the $K$-orbits of $X_\omega$ are exactly the subsets $\mathcal{O}_{(w,\varepsilon)}^{\eta,\epsilon}$.
On the other hand the subsets $\mathfrak{O}_{(w,\varepsilon)}\cap X_\omega$ (for $(w,\varepsilon)\in\mathfrak{I}_n(p,q)$) are $G^0$-stable and pairwise disjoint. By Matsuki duality there is a bijection between $K$-orbits and $G^0$-orbits. This forces $\mathfrak{O}_{(w,\varepsilon)}^{\eta,\epsilon}=\mathfrak{O}_{(w,\varepsilon)}\cap X_\omega$ to be a single $G^0$-orbit whenever $(w,\varepsilon)\in\mathfrak{I}_n^{\eta,\epsilon}(p,q)$ and  $\mathfrak{O}_{(w,\varepsilon)}\cap X_\omega$ to be empty if $(w,\varepsilon)\notin\mathfrak{I}_n^{\eta,\epsilon}(p,q)$.
This proves Proposition \ref{P3}\,{\rm (b)}.

Since the orbits $\mathcal{O}_{(w,\varepsilon)},\mathfrak{O}_{(w,\varepsilon)}\subset X$ are Matsuki dual (see Proposition \ref{P2}\,{\rm (d)}), their intersection $\mathcal{O}_{(w,\varepsilon)}\cap\mathfrak{O}_{(w,\varepsilon)}$ is compact, hence such is the intersection $\mathcal{O}_{(w,\varepsilon)}^{\eta,\epsilon}\cap
\mathfrak{O}_{(w,\varepsilon)}^{\eta,\epsilon}$ for all $(w,\varepsilon)\in\mathfrak{I}_n^{\eta,\epsilon}(p,q)$.  This implies that $\mathcal{O}_{(w,\varepsilon)}^{\eta,\epsilon}$ and $\mathfrak{O}_{(w,\varepsilon)}^{\eta,\epsilon}$ are Matsuki dual (see \cite{Gindikin-Matsuki}), and therefore part {\rm (c)} of the statement.

Let $(w,\varepsilon)\in\mathfrak{I}_n^{\eta,\epsilon}(p,q)$. Since $\mathcal{O}_{(w,\varepsilon)}^{\eta,\epsilon}$ and $\mathfrak{O}_{(w,\varepsilon)}^{\eta,\epsilon}$ are Matsuki dual, their intersection is a single $K\cap G^0$-orbit. The set on the left-hand side in (\ref{proof.P3.3}) is nonempty (by Claim 2) and $K\cap G^0$-stable, hence equality holds in (\ref{proof.P3.3}). Similarly, the sets on the left-hand sides in (\ref{proof.P3.1}) and (\ref{proof.P3.2}) are nonempty (by Claim 2) and respectively $K$- and $G^0$-stable. Since $\mathcal{O}_{(w,\varepsilon)}^{\eta,\epsilon}=\mathcal{O}_{(w,\varepsilon)}\cap X_\omega$ and $\mathfrak{O}_{(w,\varepsilon)}^{\eta,\epsilon}=\mathfrak{O}_{(w,\varepsilon)}\cap X_\omega$ are respectively a $K$-orbit and a $G^0$-orbit, equality holds in (\ref{proof.P3.1}) and (\ref{proof.P3.2}). This shows part {\rm (a)} of the statement.

Thus the proof of Proposition \ref{P3} will be complete once we establish Claims 1 and 2.

\medskip
\noindent
{\it Proof of Claim 1.}
Note that for two subspaces $A,B\subset V$ we have
$A^\perp+B^\perp=(A\cap B)^\perp$, hence
\begin{equation}
\label{proof-P3.4}
\dim A^\perp\cap B^\perp+\dim A+\dim B=\dim A\cap B+\dim V.
\end{equation}
Note also that the map $\delta$ is selfadjoint (in types BD1 and C2) or antiadjoint (in types C1 and D3) with respect to $\omega$, hence the equality $\delta(A)^\perp=\delta(A^\perp)$ holds for any subspace $A\subset V$ in all types.

Let $(w,\varepsilon)\in\mathfrak{I}_n(p,q)$ such that $\mathcal{O}_{(w,\varepsilon)}\cap X_\omega\not=\emptyset$.
Let $\mathcal{F}=(F_0,\ldots,F_n)\in \mathcal{O}_{(w,\varepsilon)}\cap X_\omega$.

By applying (\ref{proof-P3.4}) to $A=\delta\left( F_k\right)$ and $B=F_\ell$ for $1\leq k,\ell\leq n$ we obtain
\begin{equation}
\label{proof-P3.5}
\dim \delta\left( F_{n-k}\right)\cap F_{n-\ell}+k+\ell=\dim \delta \left(F_k\right)\cap F_\ell+n.
\end{equation}
On the other hand since $\mathcal{F}\in\mathcal{O}_{(w,\varepsilon)}$ Proposition \ref{P2}\,{\rm (b)} gives
\begin{equation}
\label{proof-P3.6}
\dim \delta\left( F_{n-k}\right)\cap F_{n-\ell}=|\{j=1,\ldots,n-\ell:1\leq w_j\leq n-k\}|
\end{equation}
and
\begin{eqnarray}
\label{proof-P3.7}
\dim\delta \left(F_k\right)\cap F_\ell & = & |\{j=1,\ldots,\ell:1\leq w_j\leq k\}| \\
\nonumber & = & \ell-|\{j=1,\ldots,\ell:w_j\geq k+1\}| \\
\nonumber & = & \ell-(n-k-|\{j\geq \ell+1:w_j\geq k+1\}|) \\
\nonumber & = & \ell+k-n+|\{j=1,\ldots,n-\ell:w_0ww_0(j)\leq n-k\}|
\end{eqnarray}
for all $k,\ell\in\{1,\ldots,n\}$. Comparing (\ref{proof-P3.5})--(\ref{proof-P3.7}) we conclude that $w=w_0ww_0$.

Let $k\in\{1,\ldots,n\}$ such that $w_k=k$. Since $ww_0=w_0w$, we have $w_{n-k+1}=n-k+1$.
Applying (\ref{proof-P3.4}) with $A=F_k$ (resp., $A=F_{k-1}$) and $B=V_+$, we get
\[1+\dim F_{k-1}\cap V_+-\dim F_k\cap V_+=\dim F_{n-k+1}\cap V_--\dim F_{n-k}\cap V_-\]
in types BD1 and C2 (where $V_+^\perp=V_-$), whence
\begin{eqnarray*}
\varepsilon_k=1 & \Leftrightarrow &  \dim F_k\cap V_+=\dim F_{k-1}\cap V_++1 \\
 & \Leftrightarrow &  \dim F_{n-k+1}\cap V_-=\dim F_{n-k}\cap V_-\Leftrightarrow\varepsilon_{n-k+1}=1
\end{eqnarray*}
in that case.  In types C1 and D3  (where $V_+^\perp=V_+$), we get
\[1+\dim F_{k-1}\cap V_+-\dim F_k\cap V_+=\dim F_{n-k+1}\cap V_+-\dim F_{n-k}\cap V_+\,,\]
whence also
\[\varepsilon_k=1\Leftrightarrow \varepsilon_{n-k+1}=-1\,.\] 

At this point we obtain that the signed involution $(w,\varepsilon)$ satisfies conditions {\rm (i)}--{\rm (ii)} in Section \ref{S3.3}.
To conclude that $(w,\varepsilon)\in\mathfrak{I}_n^{\eta,\epsilon}(p,q)$, it remains to check that in types C2 and D3 we have $w_k\not=n-k+1$ for all $k\leq\frac{n}{2}$. Arguing by contradiction, assume that $w_k=n-k+1$. Since $\mathcal{F}\in\mathcal{O}_{(w,\varepsilon)}$ there is a $(w,\varepsilon)$-conjugate basis $\underline{v}=(v_1,\ldots,v_n)$ such that $\mathcal{F}=\mathcal{F}(\underline{v})$.
Thus $\delta(v_k)=v_{n-k+1}$ so that we can write $v_k=v_k^++v_k^-$ and $v_{n-k+1}=v_k^+-v_k^-$.
In type C2 we have $V_+^\perp=V_-$ and $\omega$ is antisymmetric, hence
\[\omega(v_k^++v_k^-,v_k^+-v_k^-)=\omega(v_k^+,v_k^+)-\omega(v_k^-,v_k^-)=0-0=0.\]
In type D3 we have $V_+^\perp=V_+$, $V_-^\perp=V_-$, and $\omega$ is symmetric hence
\[\omega(v_k^++v_k^-,v_k^+-v_k^-)=-\omega(v_k^+,v_k^-)+\omega(v_k^-,v_k^+)=0.\]
In both cases we deduce
\[F_{n-k+1}=F_{n-k}+\langle v_{n-k+1}\rangle_{\mathbb{C}}\subset F_k^\perp+F_{k-1}^\perp\cap\langle v_k\rangle_\mathbb{C}^\perp=F_k^\perp=F_{n-k},\]
a contradiction.
This completes the proof of Claim 1. 

\medskip
\noindent
{\it Proof of Claim 2.}
For $k\in\{1,\ldots,n\}$ set $k^*=n-k+1$. We can write
\[w=(c_1;c'_1)\cdots(c_s;c'_s)(c'^*_1;c^*_1)\cdots(c'^*_s;c^*_s)(d_1;d^*_1)\cdots(d_t;d^*_t)\]
where $c_1<\ldots<c_s<c^*_s<\ldots<c^*_1$, $c_j<c'_j\not=c^*_j$ for all $j$, $d_1<\ldots<d_t<d^*_t<\ldots<d^*_1$. Note that $t=0$ in types C2 and D3.
Moreover, we denote
\begin{eqnarray*}
 & \{a_1<\ldots<a_{p-t-2s}\}:=\{k:w_k=k,\ \varepsilon_k=1\},\\
 & \{b_1<\ldots<b_{q-t-2s}\}:=\{k:w_k=k,\ \varepsilon_k=-1\}.
\end{eqnarray*}

We can construct a $\phi$-orthonormal basis
\[x_1^+,\ldots,x_t^+,y^+_1,\ldots,y^+_s,y^{+*}_s,\ldots,y^{+*}_1,z^+_1,\ldots,z^+_{p-t-2s}\]
of $V_+$, and a $(-\phi)$-orthonormal basis
\[x_1^-,\ldots,x_t^-,y^-_1,\ldots,y^-_s,y^{-*}_s,\ldots,y^{-*}_1,z^-_1,\ldots,z^-_{q-t-2s}\]
of $V_-$,
such that in types BD1 and C2 (where the restriction of $\omega$ on $V_+$ and $V_-$ is nondegenerate) we have
\begin{eqnarray*}
 & \omega(x^+_j,x^+_j)=\omega(x^-_j,x^-_j)=1,\\ 
 & \omega(y^+_j,y^{+*}_j)=\omega(y^-_j,y^{-*}_j)=1,\quad\omega(y^{+*}_j,y^+_j)=\omega(y^{-*}_j,y^-_j)=\epsilon, \\
 & \omega(z^+_j,z^+_{\ell})=\left\{\begin{array}{ll} 1 & \mbox{if $j\leq \ell=p-t-2s+1-j$} \\ \epsilon & \mbox{if $j>\ell=p-t-2s+1-j, $} \end{array}\right.
\\
 & \omega(z^-_j,z^-_{\ell})=\left\{\begin{array}{ll} 1 & \mbox{if $j\leq \ell=q-t-2s+1-j$} \\ \epsilon & \mbox{if $j>\ell=q-t-2s+1-j, $} \end{array}\right.
\end{eqnarray*}
and the other values of $\omega$ on the basis to equal $0$.  In types C1 and D3 (where $V_+^\perp=V_+$, $V_-^\perp=V_-$, and in particular $p=q=\frac{n}{2}$ in this case) we require that
\begin{eqnarray*}
 & \omega(x^+_j,x^-_j)=i,\quad \omega(x^-_j,x^+_j)=\epsilon i,\\ 
 & \omega(y^+_j,y^{-*}_j)=\omega(y^-_j,y^{+*}_j)=1,\quad\omega(y^{+*}_j,y^-_j)=\omega(y^{-*}_j,y^+_j)=\epsilon, \\
 & \omega(z^+_j,z^-_{\ell})=\epsilon\omega(z^-_\ell,z^+_j)=\left\{\begin{array}{ll} 1 & \mbox{if $\ell=\tilde{j}:=\frac{n}{2}-t-2s+1-j$ and $a_j<b_{\tilde{j}}$} \\ \epsilon & \mbox{if $\ell=\tilde{j}:=\frac{n}{2}-t-2s+1-j$ and $a_j>b_{\tilde{j}}$,} \end{array}\right.
\end{eqnarray*}
while the other values of $\omega$ on the basis are $0$.
In contrast to the value of $\omega(z_j^\pm,z_\ell^\pm)$ in types BD1,C2, the value of $\omega(z_j^+,z_\ell^-)$ in types C1,D3 is not subject to a constraint but is chosen so that the basis $(v_1,\ldots,v_n)$ below satisfies (\ref{3.3.18}).

In all cases we construct a basis $(v_1,\ldots,v_n)$ by setting
\[v_{d_j}=\frac{x^+_j+ix^-_j}{\sqrt{2}},\quad v_{d_j^*}=\frac{x^+_j-ix^-_j}{\sqrt{2}},\]
\[v_{c_j}=\frac{y^+_j+y^-_j}{\sqrt{2}},\quad v_{c'_j}=\frac{y^+_j-y^-_j}{\sqrt{2}},\quad v_{c^*_j}=\frac{y^{+*}_j+y^{-*}_j}{\sqrt{2}},\quad v_{c'^*_j}=\frac{y^{+*}_j-y^{-*}_j}{\sqrt{2}},\]
\[v_{a_j}=z^+_j,\quad\mbox{and}\quad v_{b_j}=z_j^-.\]
It is straightforward to check that the basis $(v_1,\ldots,v_n)$ is both $(w,\varepsilon)$-dual and $(w,\varepsilon)$-conjugate and satisfies (\ref{3.3.18}). This completes the proof of Claim 2.
\end{proof}

\section{Orbit duality in ind-varieties of generalized flags}

\label{S4}

Following the pattern of Section~\ref{S3}, we now present our results on orbit duality in the infinite-dimensional case.  All proofs are given in Section~\ref{S4.proofs}.
\subsection{Types A1 and A2}

\label{S4.1}

The notation is as Section \ref{S:A1A2}.
For every $\ell\in\NN$ there is a unique $\ell^*\in\NN$
such that $\omega(e_\ell,e_{\ell^*})\not=0$,
and this yields a bijection $\iota:\NN\to\NN$, $\ell\mapsto\ell^*$.

Let $\mathfrak{I}_\infty(\iota)$ be the set of involutions $w:\NN\to\NN$
such that $w(\ell)=\ell^*$ for all but finitely many $\ell\in\NN$.
In particular we have $w\iota\in\mathfrak{S}_\infty$ for all $w\in\mathfrak{I}_\infty(\iota)$.
Let $\mathfrak{I}'_\infty(\iota)\subset \mathfrak{I}_\infty(\iota)$
be the subset of involutions without fixed points (i.e., such that $w(\ell)\not=\ell$ for all $\ell\in\NN$).

Let $\sigma:\NN\to(A,\prec)$ be a bijection onto a totally ordered set,
and let us consider the ind-variety of generalized flags $\mathbf{X}(\mathcal{F}_\sigma,E)$.
In Proposition \ref{P4-1.1} below we show that the $\mathbf{K}$-orbits 
and the $\mathbf{G}^0$-orbits of $\mathbf{X}(\mathcal{F}_\sigma,E)$ are parametrized by the elements of $\mathfrak{I}_\infty(\iota)$ in type A1, and by elements of $\mathfrak{I}'_\infty(\iota)$ in type A2.

\begin{definition}
Let $w\in\mathfrak{I}_\infty(\iota)$.
Let $\underline{v}=(v_1,v_2,\ldots)$ be a basis of $\mathbf{V}$ such that
\begin{equation}
\label{4.1.1}
v_\ell=e_\ell\quad\mbox{for all but finitely many $\ell\in\NN$}.
\end{equation}
We call $\underline{v}$ {\em $w$-dual} if in addition to (\ref{4.1.1}) $\underline{v}$ satisfies
\[\omega(v_\ell,v_k)=\left\{
\begin{array}{ll}
0 & \mbox{if $\ell\not=w_k$,} \\
\pm1 & \mbox{if $\ell=w_k$}
\end{array}
\quad\mbox{for all $k,\ell\in\NN$,}
\right.\]
and we call $\underline{v}$ {\em $w$-conjugate} if in addition to (\ref{4.1.1}) 
\[\gamma(v_k)=\pm v_{w_k}\quad\mbox{for all $k\in\NN$.}\]
Set
$\bs{\mathcal{O}}_w:=\{\mathcal{F}_{\sigma}(\underline{v}):\mbox{$\underline{v}$ is $w$-dual}\}$ and $\bs{\mathfrak{O}}_w:=\{\mathcal{F}_{\sigma}(\underline{v}):\mbox{$\underline{v}$ is $w$-conjugate}\}$, so that $\bs{\mathcal{O}}_w$ and $\bs{\mathfrak{O}}_w$ are subsets of the ind-variety $\mathbf{X}(\mathcal{F}_\sigma,E)$.
\end{definition}

\paragraph{\bf Notation.} 
{\rm (a)}
We use the abbreviation $\mathbf{X}:=\mathbf{X}(\mathcal{F}_\sigma,E)$. \\
{\rm (b)}
If $\mathcal{F}$ is a generalized flag weakly compatible with $E$,
then $\mathcal{F}^\perp:=\{F^\perp:F\in\mathcal{F}\}$ is also a generalized flag weakly compatible with $E$.

Let $(A^*,\prec^*)$ be the totally ordered set given by $A^*=A$ as a set and $a\prec^*a'$ whenever $a\succ a'$.
Let $\sigma^\perp:\NN\to(A^*,\prec^*)$ be defined by $\sigma^\perp(\ell)=\sigma(\ell^*)$.
Then we have
$\mathcal{F}_\sigma^\perp=\mathcal{F}_{\sigma^\perp}$.
Note that $\mathcal{F}^\perp$ is $E$-commensurable with $\mathcal{F}_{\sigma^\perp}$
whenever $\mathcal{F}$ is $E$-commensurable with $\mathcal{F}_\sigma$.
Hence the map 
\[\mathbf{X}\to\mathbf{X}^\perp:=\mathbf{X}(\mathcal{F}_{\sigma^\perp},E),\ \mathcal{F}\mapsto\mathcal{F}^\perp\]
is well defined.
We also use the abbreviation $\pmb{\mathbb{O}}^\perp_w:=(\pmb{\mathbb{O}}_{\sigma^\perp,\sigma})_w$ for all $w\in\mathfrak{S}_\infty$. \\
{\rm (c)}
We further note that $\gamma(\mathcal{F}_\sigma)=\mathcal{F}_{\sigma\circ\iota}$
and that $\gamma(\mathcal{F})\in\mathbf{X}^\gamma:=\mathbf{X}(\mathcal{F}_{\sigma\circ\iota},E)$ whenever $\mathcal{F}\in\mathbf{X}$.
We also abbreviate  $\pmb{\mathbb{O}}^\gamma_w:=(\pmb{\mathbb{O}}_{\sigma\circ\iota,\sigma})_w$ for all $w\in\mathfrak{S}_\infty$.

\medskip
Thus
$\displaystyle \mathbf{X}^\perp\times\mathbf{X}=\bigsqcup_{w\in\mathfrak{S}_\infty}\pmb{\mathbb{O}}^\perp_w$
and
$\displaystyle \mathbf{X}^\gamma\times\mathbf{X}=\bigsqcup_{w\in\mathfrak{S}_\infty}\pmb{\mathbb{O}}^\gamma_w$
(see Proposition \ref{P2-3.2}).

\begin{proposition}
\label{P4-1.1}
Let $\mathfrak{I}^\epsilon_\infty(\iota)=\mathfrak{I}_\infty(\iota)$ in type A1
and $\mathfrak{I}^\epsilon_\infty(\iota)=\mathfrak{I}'_\infty(\iota)$ in type A2.
\begin{itemize}
\item[\rm (a)] For every $w\in\mathfrak{I}^\epsilon_\infty(\iota)$, \[\bs{\mathcal{O}}_w\cap\bs{\mathfrak{O}}_w=\{\mathcal{F}_\sigma(\underline{v}):\mbox{$\underline{v}$ is $w$-dual and $w$-conjugate}\}\not=\emptyset.\]
\item[\rm (b)] For every $w\in\mathfrak{I}^\epsilon_\infty(\iota)$,
\[
\bs{\mathcal{O}}_w=\{\mathcal{F}\in\mathbf{X}:(\mathcal{F}^\perp,\mathcal{F})\in\pmb{\mathbb{O}}^\perp_{w\iota}\} \quad\mbox{and}\quad
\bs{\mathfrak{O}}_w=\{\mathcal{F}\in\mathbf{X}:(\gamma(\mathcal{F}),\mathcal{F})\in\pmb{\mathbb{O}}^\gamma_{w\iota}\}.
\]
\item[\rm (c)] The subsets $\bs{\mathcal{O}}_w$ (for $w\in\mathfrak{I}^\epsilon_\infty(\iota)$) are exactly the $\mathbf{K}$-orbits of $\mathbf{X}$.
The subsets $\bs{\mathfrak{O}}_w$ (for $w\in\mathfrak{I}^\epsilon_\infty(\iota)$) are exactly the $\mathbf{G}^0$-orbits of $\mathbf{X}$.
Moreover $\bs{\mathcal{O}}_w\cap\bs{\mathfrak{O}}_w$ is a single $\mathbf{K}\cap\mathbf{G}^0$-orbit.
\end{itemize}
\end{proposition}

\subsection{Type A3}

\label{S4.2}

The notation is as in Section \ref{S:A3}.
In particular, we fix a partition $\NN=N_+\sqcup N_-$ yielding $\Phi$ as in (\ref{phi}) and we consider the corresponding
hermitian form $\phi$ and involution $\delta$ on $\mathbf{V}$.

Let $\mathfrak{I}_\infty(N_+,N_-)$ be the set of pairs $(w,\varepsilon)$ consisting of an involution $w:\NN\to\NN$ and a map $\varepsilon:\{\ell:w_\ell=\ell\}\to\{1,-1\}$ such that the subsets
\[N'_\pm=N'_\pm(w,\varepsilon):=\{\ell\in N_\pm:(w_\ell,\varepsilon_\ell)=(\ell,\pm1)\}
\]
satisfy
\[|N_\pm\setminus N'_\pm|=|\{\ell\in N_\mp:(w_\ell,\varepsilon_\ell)=(\ell,\pm 1)\}|+\frac{1}{2}|\{\ell\in\NN:w_\ell\not=\ell\}|<\infty.\]
In particular, $w\in\mathfrak{S}_\infty$.

Fix $\sigma:\NN\to (A,\prec)$ a bijection onto a totally ordered set.
We show in Proposition \ref{P4.2-2} that the $\mathbf{K}$-orbits and the $\mathbf{G}^0$-orbits of the ind-variety $\mathbf{X}:=\mathbf{X}(\mathcal{F}_\sigma,E)$ are parametrized by the elements of $\mathfrak{I}_\infty(N_+,N_-)$.

\begin{definition}
Let $(w,\varepsilon)\in\mathfrak{I}_\infty(N_+,N_-)$.
A basis $\underline{v}=(v_1,v_2,\ldots)$ of $\mathbf{V}$ such that $v_\ell=e_\ell$ for all but finitely many $\ell\in\NN$ is {\em $(w,\varepsilon)$-conjugate} if
\[\delta(v_k)=\left\{
\begin{array}{ll}
v_{w_k} & \mbox{if $w_k\not=k$,} \\
\varepsilon_kv_k & \mbox{if $w_k=k$}
\end{array}
\right.
\quad\mbox{for all $k\in\NN$,}\]
and is {\em $(w,\varepsilon)$-dual} if
\[
\phi(v_k,v_\ell)=\left\{
\begin{array}{ll}
0 & \mbox{if $\ell\not=w_k$,} \\
1 & \mbox{if $\ell=w_k\not=k$,} \\
\varepsilon_k & \mbox{if $\ell=w_k=k$}
\end{array}
\right.
\quad\mbox{for all $k,\ell\in\NN$}.
\]
Set
$\bs{\mathcal{O}}_{(w,\varepsilon)}:=\{\mathcal{F}_\sigma(\underline{v}):\mbox{$\underline{v}$ is $(w,\varepsilon)$-conjugate}\}$, $\bs{\mathfrak{O}}_{(w,\varepsilon)}:=\{\mathcal{F}_\sigma(\underline{v}):\mbox{$\underline{v}$ is $(w,\varepsilon)$-dual}\}$.

\end{definition}

\paragraph{\bf Notation}
{\rm (a)} Note that every subspace in the generalized flag $\mathcal{F}_\sigma$ is $\delta$-stable, i.e., $\delta(\mathcal{F}_\sigma)=\mathcal{F}_\sigma$.
The map $\mathbf{X}\to\mathbf{X}$, $\mathcal{F}\mapsto\delta(\mathcal{F})$ is well defined.
\\
{\rm (b)} Write $F^\dag=\{x\in\mathbf{V}:\phi(x,y)=0\ \forall y\in F\}$ and
$\mathcal{F}^\dag:=\{F^\dag:F\in \mathcal{F}\}$, which is a generalized flag weakly compatible with $E$ whenever $\mathcal{F}$ is so.

As in Section \ref{S4.1} we write $(A^*,\prec^*)$ for the totally ordered set
such that $A^*=A$ and $a\prec^*a'$ whenever $a\succ a'$.
It is readily seen that $\mathcal{F}_\sigma^\dag=\mathcal{F}_{\sigma^\dag}$
where $\sigma^\dag:\NN\to(A^*,\prec^*)$ is such that $\sigma^\dag(\ell)=\sigma(\ell)$ for all $\ell\in\NN$, and we get a well-defined map
\[\mathbf{X}\to\mathbf{X}^\dag:=\mathbf{X}(\mathcal{F}_{\sigma^\dag},E),\ \mathcal{F}\mapsto\mathcal{F}^\dag.\]
{\rm (c)}
We write $\pmb{\mathbb{O}}_w:=(\pmb{\mathbb{O}}_{\sigma,\sigma})_w$ and
$\pmb{\mathbb{O}}_w^\dag:=(\pmb{\mathbb{O}}_{\sigma^\dag,\sigma})_w$ so that
\[\mathbf{X}\times\mathbf{X}=\bigsqcup_{w\in\mathfrak{S}_\infty}\pmb{\mathbb{O}}_w
\quad\mbox{and}\quad
\mathbf{X}^\dag\times\mathbf{X}=\bigsqcup_{w\in\mathfrak{S}_\infty}\pmb{\mathbb{O}}^\dag_w
\]
(see Proposition \ref{P2-3.2}).

\begin{proposition}
\label{P4.2-2}
\begin{itemize}
\item[\rm (a)] For every $(w,\varepsilon)\in\mathfrak{I}_\infty(N_+,N_-)$ we have
\[\bs{\mathcal{O}}_{(w,\varepsilon)}\cap\bs{\mathfrak{O}}_{(w,\varepsilon)}=\{\mathcal{F}_\sigma(\underline{v}):\mbox{$\underline{v}$ is $(w,\varepsilon)$-conjugate and $(w,\varepsilon)$-dual}\}\not=\emptyset.\]
\item[\rm (b)] Let $(w,\varepsilon)\in\mathfrak{I}_\infty(N_+,N_-)$
and $\mathcal{F}=\{F'_a,F''_a:a\in A\}\in\mathbf{X}$. Then
$\mathcal{F}\in\bs{\mathcal{O}}_{(w,\varepsilon)}$ (resp., $\mathcal{F}\in\bs{\mathfrak{O}}_{(w,\varepsilon)}$) if and only if
\[(\delta(\mathcal{F}),\mathcal{F})\in\pmb{\mathbb{O}}_w\quad\mbox{(resp., $(\mathcal{F}^\dag,\mathcal{F})\in\pmb{\mathbb{O}}_w^\dag$)}\]
and
\[
\dim F''_{\sigma(\ell)}\cap \mathbf{V}_\pm/F'_{\sigma(\ell)}\cap \mathbf{V}_\pm
=
\left\{\begin{array}{ll}
1 & \mbox{if $\sigma(w_\ell)\prec\sigma(\ell)$ or $(w_\ell,\varepsilon_\ell)=(\ell,\pm1)$,} \\
0 & \mbox{otherwise}
\end{array}\right. 
\]
where $\mathbf{V}_\pm=\langle e_\ell:\ell\in N_\pm\rangle_\mathbb{C}$
(resp., for $n\in\NN$ large enough
\[
\varsigma(\phi:F''_{\sigma(\ell)}\cap V_n)=\varsigma(\phi:F'_{\sigma(\ell)}\cap V_n)+\left\{
\begin{array}{ll}
(1,1) & \mbox{if $\sigma(w_\ell)\prec\sigma(\ell)$,} \\
(1,0) & \mbox{if $(w_\ell,\varepsilon_\ell)=(\ell,1)$,} \\
(0,1) & \mbox{if $(w_\ell,\varepsilon_\ell)=(\ell,-1)$,} \\
(0,0) & \mbox{if $\sigma(w_\ell)\succ\sigma(\ell)$}
\end{array}
\right.
\]
where we 
$V_n=\langle e_k:k\leq n\rangle_\mathbb{C}$
and
$\varsigma(\phi:F)$ stands for the signature of $\phi$ on $F/F\cap F^\dag$)
for all $\ell\in\NN$.
\item[\rm (c)] The subsets $\bs{\mathcal{O}}_{(w,\varepsilon)}$ ($(w,\varepsilon)\in\mathfrak{I}_\infty(N_+,N_-)$) are exactly the $\mathbf{K}$-orbits of $\mathbf{X}$. The subsets $\bs{\mathfrak{O}}_{(w,\varepsilon)}$ ($(w,\varepsilon)\in\mathfrak{I}_\infty(N_+,N_-)$) are exactly the $\mathbf{G}^0$-orbits of $\mathbf{X}$.
Moreover $\bs{\mathcal{O}}_{(w,\varepsilon)}\cap\bs{\mathfrak{O}}_{(w,\varepsilon)}$ is a single $\mathbf{K}\cap\mathbf{G}^0$-orbit.
\end{itemize}
\end{proposition}

\subsection{Types B, C, D}

\label{S4.3}

Assume that $\mathbf{V}$ is endowed with a nondegenerate symmetric or symplectic form $\omega$, determined by a matrix $\Omega$ as in (\ref{omega}).
Let $\iota:\NN\to\NN$, $\ell\mapsto\ell^*$ 
satisfy $\omega(e_\ell,e_{\ell^*})\not=0$ for all $\ell$.

Let $\NN=N_+\sqcup N_-$ be a partition such that $N_+,N_-$ are either both $\iota$-stable or such that $\iota(N_+)=N_-$.
As before, let $\phi$ and $\delta$ be the hermitian form and the involution of $\mathbf{V}$
corresponding to this partition.
The following table summarizes the different cases.
\begin{center}
\begin{tabular}{|c|c|c|}
\cline{2-3}
\multicolumn{1}{c|}{ } & \begin{tabular}{c} $\omega$ symmetric \\ $\epsilon=1$ \end{tabular} & \begin{tabular}{c} $\omega$ symplectic \\ $\epsilon=-1$ \end{tabular} \\
\hline
\begin{tabular}{c} $\iota(N_\pm)\subset N_\pm$ \\ $\eta=1$ \end{tabular} & type BD1 & type C2 \\
\hline
\begin{tabular}{c} $\iota(N_\pm)=N_\mp$ \\ $\eta=-1$ \end{tabular} & type D3 & type C1 \\
\hline
\end{tabular}
\end{center}

Let $\mathfrak{I}_\infty^{\eta,\epsilon}(N_+,N_-)\subset
\mathfrak{I}_\infty(N_+,N_-)$ be the subset of pairs
$(w,\varepsilon)$ such that 
\begin{itemize}
\item[\rm (i)] $\iota w=w\iota$ (hence the set $\{\ell:w_\ell=\ell\}$ is $\iota$-stable);
\item[\rm (ii)] $\varepsilon_{\iota(k)}=\eta\varepsilon_k$ for all $k\in\{\ell:w_\ell=\ell\}$;
\item[\rm (iii)] and if $\eta\epsilon=-1$: $w_k\not=\iota(k)$ for all $k\in\mathbb{N}^*$.
\end{itemize}

Let $\mathcal{F}_\sigma$ be an $\omega$-isotropic maximal generalized flag compatible with $E$.
Thus $\sigma:\NN\to(A,\prec)$ is a bijection
onto a totally ordered set $(A,\prec)$ endowed with an (involutive) antiautomorphism of ordered sets
$\iota_A:(A,\prec)\to (A,\prec)$ such that $\sigma\iota=\iota_A\sigma$.
The following statement shows that the $\mathbf{K}$-orbits and the $\mathbf{G}^0$-orbits
of the ind-variety $\mathbf{X}_\omega:=\mathbf{X}_\omega(\mathcal{F}_\sigma,E)$
are parametrized by the elements of the set $\mathfrak{I}_\infty^{\eta,\epsilon}(N_+,N_-)$.

\begin{proposition}
\label{P4.3}
We consider bases $\underline{v}=(v_1,v_2,\ldots)$ of $\mathbf{V}$ such that 
\begin{equation}
\label{4-3.1}
\omega(v_k,v_\ell)\not=0\quad\mbox{if and only if}\quad \ell=\iota(k).
\end{equation}
\begin{itemize}
\item[\rm (a)] For every $(w,\varepsilon)\in\mathfrak{I}^{\eta,\epsilon}_\infty(N_+,N_-)$ we have
\begin{eqnarray*}
 & \bs{\mathcal{O}}_{(w,\varepsilon)}^{\eta,\epsilon}:=\bs{\mathcal{O}}_{(w,\varepsilon)}\cap\mathbf{X}_\omega=\{\mathcal{F}_\sigma(\underline{v}):\mbox{$\underline{v}$ is $(w,\varepsilon)$-conjugate and satisfies (\ref{4-3.1})}\}\not=\emptyset, \\
 & \bs{\mathfrak{O}}_{(w,\varepsilon)}^{\eta,\epsilon}:=\bs{\mathfrak{O}}_{(w,\varepsilon)}\cap\mathbf{X}_\omega=\{\mathcal{F}_\sigma(\underline{v}):\mbox{$\underline{v}$ is $(w,\varepsilon)$-dual and satisfies (\ref{4-3.1})}\}\not=\emptyset, \\
 & \bs{\mathcal{O}}_{(w,\varepsilon)}^{\eta,\epsilon}\cap\bs{\mathfrak{O}}_{(w,\varepsilon)}^{\eta,\epsilon}=\{\mathcal{F}_\sigma(\underline{v}):\mbox{$\underline{v}$ is $(w,\varepsilon)$-conjugate, $(w,\varepsilon)$-dual and satisfies (\ref{4-3.1})}\}\not=\emptyset.
\end{eqnarray*}
\item[\rm (b)] The subsets $\bs{\mathcal{O}}_{(w,\varepsilon)}^{\eta,\epsilon}$
($(w,\varepsilon)\in\mathfrak{I}_\infty^{\eta,\epsilon}(N_+,N_-)$)
are exactly the $\mathbf{K}$-orbits of $\mathbf{X}_\omega$.
The subsets $\bs{\mathfrak{O}}_{(w,\varepsilon)}^{\eta,\epsilon}$
($(w,\varepsilon)\in\mathfrak{I}_\infty^{\eta,\epsilon}(N_+,N_-)$)
are exactly the $\mathbf{G}^0$-orbits of $\mathbf{X}_\omega$.
Moreover $\bs{\mathcal{O}}_{(w,\varepsilon)}^{\eta,\epsilon}\cap\bs{\mathfrak{O}}_{(w,\varepsilon)}^{\eta,\epsilon}$ is a single $\mathbf{K}\cap\mathbf{G}^0$-orbit.
\end{itemize}
\end{proposition}

\subsection{Ind-variety structure}

\label{S4.4}

In this section we recall from \cite{Dimitrov-Penkov} the ind-variety structure on $\mathbf{X}$ and $\mathbf{X}_\omega$.

Recall that $E=(e_1,e_2,\ldots)$ is a (countable) basis of $\mathbf{V}$.
Fix  an $E$-compatible maximal generalized flag $\mathcal{F}_\sigma$ corresponding
to a bijection $\sigma:\NN\to(A,\prec)$ onto a totally ordered set, and let
$\mathbf{X}=\mathbf{X}(\mathcal{F}_\sigma,E)$.

Let $V_n:=\langle e_1,\ldots,e_n\rangle_\mathbb{C}$ and
let $X_n$ denote the variety of complete flags of $V_n$ defined as in (\ref{2.2.X}).
There are natural inclusions
$V_n\subset V_{n+1}$
and
\begin{equation}
\label{4.4-1}
\mathrm{GL}(V_n)\cong\{g\in \mathrm{GL}(V_{n+1}):g(V_n)=V_n,\ g(e_{n+1})=e_{n+1}\}\subset \mathrm{GL}(V_{n+1}),
\end{equation}
and we obtain a $\mathrm{GL}(V_n)$-equivariant embedding
\[\iota_n=\iota_n(\sigma):X_n\to X_{n+1},\ (F_k)_{k=0}^n\mapsto (F'_k)_{k=0}^{n+1}\]
by letting
\[F'_k:=\left\{
\begin{array}{ll}
F_k & \mbox{if $a_k\prec\sigma(n+1)$} \\
F_{k-1}\oplus\langle e_{n+1}\rangle_\mathbb{C} & \mbox{if $a_k\succeq \sigma(n+1)$}
\end{array}
\right.\]
where $a_1\prec a_2\prec\ldots\prec a_{n+1}$ are the elements of the set $\{\sigma(\ell):1\leq \ell\leq n+1\}$ written in increasing order.
Therefore, we get a chain of embeddings (which are morphisms of algebraic varieties)
\[\cdots \hookrightarrow X_{n-1} \stackrel{\iota_{n-1}}{\hookrightarrow} X_n \stackrel{\iota_{n}}{\hookrightarrow} X_{n+1} \stackrel{\iota_{n+1}}{\hookrightarrow} \cdots\]
and $\mathbf{X}$ is obtained as the direct limit 
\[\mathbf{X}=\mathbf{X}(\mathcal{F}_\sigma,E)=\lim_\to X_n.\]
In particular for each $n$ we get an embedding $\hat{\iota}_n:X_n\hookrightarrow \mathbf{X}$ 
and up to identifying $X_n$ with its image by this embedding we can view $\mathbf{X}$ as the union
$\mathbf{X}=\bigcup_{n\geq 1}X_n$.
Every generalized flag $\mathcal{F}\in\mathbf{X}$ belongs to all $X_n$ after some rank $n_\mathcal{F}$. For instance $\mathcal{F}_\sigma\in X_n$ for all $n\geq 1$.

A basis $\underline{v}=(v_1,\ldots,v_n)$ of $V_n$ can be completed into the basis of $\mathbf{V}$ denoted by $\underline{\hat{v}}:=(v_1,\ldots,v_n,e_{n+1},e_{n+2},\ldots)$, and we have
\begin{equation}
\label{completed-basis}
\hat{\iota}_n(\mathcal{F}(v_{\tau_1},\ldots,v_{\tau_n}))=\mathcal{F}_\sigma(\underline{\hat{v}})
\end{equation}
(using the notation of Sections \ref{S2.2}--\ref{S2.3})
where $\tau=\tau^{(n)}\in\mathfrak{S}_n$ is the permutation such that
$\sigma(\tau^{(n)}_1)\prec\ldots\prec\sigma(\tau^{(n)}_n)$.

Recall that the ind-topology on $\mathbf{X}$ is defined by declaring a subset $\mathbf{Z}\subset\mathbf{X}$ open (resp., closed) if every intersection $\mathbf{Z}\cap X_n$ is open (resp., closed).

Clearly the ind-variety structure on $\mathbf{X}$ is not modified if the sequence $(X_n,\iota_n)_{n\geq 1}$ is replaced by a subsequence $(X_{n_k},\iota'_k)_{k\geq 1}$
where $\iota'_k:=\iota_{n_{k+1}-1}\circ\cdots\circ\iota_{n_k}$.

In type A3 (using the notation of Section \ref{S2.1}) the subspace $V_n\subset\mathbf{V}$ is endowed with the restrictions of $\phi$ and $\delta$
hence we can define $K_n,G_n^0\subset \mathrm{GL}(V_n)$ (as in Section \ref{S3.2})
and the inclusion of (\ref{4.4-1}) restricts to natural inclusions
$K_n\subset K_{n+1}$ and $G_n^0\subset G_{n+1}^0$.

Next assume that the space $\mathbf{V}$ is endowed with a nondegenerate symmetric or symplectic form $\omega$ determined by the matrix $\Omega$ of (\ref{omega}).
The blocks $J_1,J_2,\ldots$ in the matrix $\Omega$ are of size $1$ or $2$.
We set $n_k:=|J_1|+\ldots+|J_k|$
so that the restriction of $\omega$ to each subspace $V_{n_k}$ is nondegenerate.
Hence in types A1, A2, BD1, C1, C2, and D3 we can define
the subgroups $K_{n_k},G_{n_k}^0\subset \mathrm{GL}(V_{n_k})$
(as in Section \ref{S3})
and (\ref{4.4-1}) yields natural inclusions
\[K_{n_k}\subset K_{n_{k+1}}\quad\mbox{and}\quad G^0_{n_k}\subset G^0_{n_{k+1}}.\]

Moreover, the subvariety $(X_{n_k})_\omega\subset X_{n_k}$ of isotropic flags (with respect to $\omega$) can be defined as in (\ref{2.2.Xomega}).
Assuming that the generalized flag $\mathcal{F}_\sigma$ is $\omega$-isotropic, the embedding 
$\iota'_k:X_{n_k}\hookrightarrow X_{n_{k+1}}$ maps $(X_{n_k})_\omega$ into $(X_{n_{k+1}})_\omega$ and we have
\[\mathbf{X}_\omega=\mathbf{X}_\omega(\mathcal{F}_\sigma,E)=\bigcup_{k\geq 1}(X_{n_k})_\omega\quad\mbox{and}\quad (X_{n_k})_\omega=\mathbf{X}_\omega\cap X_{n_k}\ \mbox{for all $k\geq 1$}.\]
In particular, $\mathbf{X}_\omega$ is a closed ind-subvariety of $\mathbf{X}$
(as stated in Proposition \ref{P2.3-3}).

\subsection{Proofs}

\label{S4.proofs}

\begin{proof}[Proof of Proposition \ref{P4-1.1}]
Let $\mathcal{F}=\{F'_a,F''_a:a\in A\}=\mathcal{F}_\sigma(\underline{v})$ for a basis $\underline{v}=(v_1,v_2,\ldots)$ of $\mathbf{V}$.
Let $w\in\mathfrak{I}_\infty^\epsilon(\iota)$.
If the basis $\underline{v}$ is $w$-dual, then 
\[
(F'_a)^\perp=\langle v_\ell:\sigma(w_\ell)\succeq a\rangle_\mathbb{C}
\quad\mbox{and}\quad
(F''_a)^\perp=\langle v_\ell:\sigma(w_\ell)\succ a\rangle_\mathbb{C}\,,
\]
hence $\mathcal{F}^\perp=\mathcal{F}_{\sigma^\perp\iota w}(\underline{v})$;
this yields
$(\mathcal{F}^\perp,\mathcal{F})\in\pmb{\mathbb{O}}^\perp_{w\iota}$.
If $\underline{v}$ is $w$-conjugate, then
\[
\gamma(F'_a)=\langle v_\ell:\sigma(w_\ell)\prec a\rangle_\mathbb{C}
\quad\mbox{and}\quad
\gamma(F''_a)=\langle v_\ell:\sigma(w_\ell)\preceq a\rangle_\mathbb{C}\,,
\]
whence $\gamma(\mathcal{F})=\mathcal{F}_{\sigma w}(\underline{v})$
and
$(\gamma(\mathcal{F}),\mathcal{F})\in\pmb{\mathbb{O}}^\gamma_{w\iota}$.
This proves the inclusions $\subset$ in Proposition \ref{P4-1.1}\,{\rm (b)}.
Note that these inclusions imply in particular that the subsets $\bs{\mathcal{O}}_w$, as well as $\bs{\mathfrak{O}}_w$, are pairwise disjoint.

For $w\in\mathfrak{I}_{n_k}^\epsilon$ we define $\hat{w}:\mathbb{N}^*\to\mathbb{N}^*$ by letting
\[\hat{w}(\ell)=\left\{
\begin{array}{ll}
\tau w\tau^{-1}(\ell) & \mbox{if $\ell\leq n_k$,} \\
\iota(\ell) & \mbox{if $\ell\geq n_k+1$}
\end{array}
\right.\]
where $\tau=\tau^{(n_k)}:\{1,\ldots,n_k\}\to\{1,\ldots,n_k\}$ is the
permutation such that $\sigma(\tau_1)\prec\ldots\prec\sigma(\tau_{n_k})$.
It is easy to see that we obtain a well-defined (injective) map
$j_k:\mathfrak{I}_{n_k}^\epsilon\to\mathfrak{I}_\infty^\epsilon(\iota)$, $j_k(w):=\hat{w}$,
and
\begin{equation}
\label{4.5.0}
\mathfrak{I}_\infty^\epsilon(\iota)=\bigcup_{k\geq 1}j_k(\mathfrak{I}_{n_k}^\epsilon).\end{equation}
Moreover,
given a basis $\underline{v}=(v_1,\ldots,v_{n_k})$ of $V_{n_k}$ and
the basis $\underline{\hat{v}}$ of $\mathbf{V}$ obtained by adding the vectors $e_\ell$ for $\ell\geq n_k+1$,
the implication
\begin{eqnarray}
\label{4.5.1}
 & \mbox{$(v_{\tau_1},\ldots,v_{\tau_{n_k}})$ is $w$-dual (resp., $w$-conjugate)} \\
\nonumber
 & \Rightarrow \
\mbox{$\underline{\hat{v}}$ is $\hat{w}$-dual (resp., $\hat{w}$-conjugate)}
\end{eqnarray}
clearly follows from our constructions.
Note that
\begin{equation}
\label{T-proof.1}
\bs{\mathcal{O}}_{\hat{w}}\cap X_{n_k}=\mathcal{O}_{w}
\quad
\mbox{and}
\quad
\bs{\mathfrak{O}}_{\hat{w}}\cap X_{n_k}=\mathfrak{O}_{w}
\end{equation}
where $\mathcal{O}_w,\mathfrak{O}_w\subset X_{n_k}$ are the orbits defined in Definition \ref{DA12};
indeed, the inclusions $\supset$ in (\ref{T-proof.1}) are implied by 
(\ref{completed-basis}) and (\ref{4.5.1}),
whereas the inclusions $\subset$ follow from 
Proposition \ref{P1}\,{\rm (c)}
and the fact that the subsets $\bs{\mathcal{O}}_{\hat{w}}$, as well as $\bs{\mathfrak{O}}_{\hat{w}}$,
are pairwise disjoint.
Parts {\rm (a)} and {\rm (c)} of Proposition \ref{P4-1.1} now follow
from (\ref{4.5.0})--(\ref{T-proof.1}) and Proposition \ref{P1}\,{\rm (a)}, {\rm (c)}. By Proposition \ref{P4-1.1}\,{\rm (a)} we deduce that equalities hold in Proposition \ref{P4-1.1}\,{\rm (b)}, and the proof is complete.
\end{proof}

\begin{proof}[Proof of Proposition \ref{P4.2-2}]
For every $n\geq 1$ we set $p_n=|N_+\cap\{1,\ldots,n\}|$ and $q_n=|N_-\cap\{1,\ldots,n\}|$.

Let $\mathcal{F}=\{F'_a,F''_a:a\in A\}=\mathcal{F}_\sigma(\underline{v})$ for some basis $\underline{v}=(v_1,v_2,\ldots)$ of $\mathbf{V}$.
Let $(w,\varepsilon)\in\mathfrak{I}_\infty(N_+,N_-)$.
If $\underline{v}$ is $(w,\varepsilon)$-conjugate, then
\[
\delta(F'_a)=\langle v_\ell:\sigma(w_\ell)\prec a\rangle_\mathbb{C}
\quad\mbox{and}\quad
\delta(F''_a)=\langle v_\ell:\sigma(w_\ell)\preceq a\rangle_\mathbb{C}
\]
so that $(\delta(\mathcal{F}),\mathcal{F})=(\mathcal{F}_{\sigma w}(\underline{v}),\mathcal{F}_\sigma(\underline{v}))\in\pmb{\mathbb{O}}_w$.
In addition,
\[\left\{
\begin{array}{ll}
\mbox{$F''_{\sigma(\ell)}\cap\mathbf{V}_+/F'_{\sigma(\ell)}\cap\mathbf{V}_+=\langle v_\ell\rangle_\mathbb{C}$, $F''_{\sigma(\ell)}\cap\mathbf{V}_-=F'_{\sigma(\ell)}\cap\mathbf{V}_-$} & \mbox{\!\!\!if $(w_\ell,\varepsilon_\ell)=(\ell,+1)$,} \\[1.5mm]
\mbox{$F''_{\sigma(\ell)}\cap\mathbf{V}_-/F'_{\sigma(\ell)}\cap\mathbf{V}_-=\langle v_\ell\rangle_\mathbb{C}$, $F''_{\sigma(\ell)}\cap\mathbf{V}_+=F'_{\sigma(\ell)}\cap\mathbf{V}_+$} & \mbox{\!\!\!if $(w_\ell,\varepsilon_\ell)=(\ell,-1)$,} \\[1.5mm]
\mbox{$F''_{\sigma(\ell)}\cap\mathbf{V}_+/F'_{\sigma(\ell)}\cap\mathbf{V}_+=\langle v_\ell+v_{w_\ell}\rangle_\mathbb{C}$\,,} \\[1mm]
\mbox{\qquad $F''_{\sigma(\ell)}\cap\mathbf{V}_-/F'_{\sigma(\ell)}\cap\mathbf{V}_-=\langle v_\ell- v_{w_\ell}\rangle_\mathbb{C}$} & \mbox{\!\!\!if $\sigma(w_\ell)\prec\sigma(\ell)$,} \\[1.5mm]
F''_{\sigma(\ell)}\cap \mathbf{V}_+=F'_{\sigma(\ell)}\cap \mathbf{V}_+,\ F''_{\sigma(\ell)}\cap \mathbf{V}_+=F'_{\sigma(\ell)}\cap \mathbf{V}_+ & \mbox{\!\!\!if $\sigma(w_\ell)\succ\sigma(\ell)$,}
\end{array}
\right.\]
which proves the formula for $\dim F''_{\sigma(\ell)}\cap \mathbf{V}_\pm/F'_{\sigma(\ell)}\cap \mathbf{V}_\pm$ stated in Proposition \ref{P4.2-2}\,{\rm (b)}.
If $\underline{v}$ is $(w,\varepsilon)$-dual, then we get similarly
\[
(F'_a)^\dag=\langle v_\ell:\sigma(w_\ell)\succeq a\rangle_\mathbb{C}
\quad\mbox{and}\quad
(F''_a)^\dag=\langle v_\ell:\sigma(w_\ell)\succ a\rangle_\mathbb{C}\,.
\]
Hence $(\mathcal{F}^\dag,\mathcal{F})=(\mathcal{F}_{\sigma^\dag w}(\underline{v}),\mathcal{F}_\sigma(\underline{v}))\in\pmb{\mathbb{O}}_w^\dag$.
For $n\geq 1$ large enough we have
$(w_\ell,\varepsilon_\ell)=(\ell,\pm1)$ for all $\ell\in N_\pm\cap\{n+1,n+2,\ldots\}$
and $v_\ell=e_\ell$ for all $\ell\geq n+1$.
Thus the pair $(\check{w},\check{\varepsilon}):=(w|_{\{1,\ldots,n\}},\varepsilon|_{\{1,\ldots,n\}})$ belongs to $\mathfrak{I}_n(p_n,q_n)$
whereas by (\ref{completed-basis}) we have
\[\mathcal{F}=\mathcal{F}(v_{\tau_1},\ldots,v_{\tau_n}).\]
The basis $(v_{\tau_1},\ldots,v_{\tau_n})$ of $V_n$ is $(\tau^{-1}\check{w}\tau,\check{\varepsilon}\tau)$-dual if $\underline{v}$ is $(w,\varepsilon)$-dual; the last formula in Proposition \ref{P4.2-2}\,{\rm (b)} now follows from Proposition \ref{P2}\,{\rm (b)} and this observation. Altogether this shows the ``only if'' part in 
Proposition \ref{P4.2-2}\,{\rm (b)}, which guarantees in particular that the subsets $\bs{\mathcal{O}}_{(w,\varepsilon)}$, as well as the subsets $\bs{\mathfrak{O}}_{(w,\varepsilon)}$,
are pairwise disjoint.
The ``if'' part of Proposition \ref{P4.2-2}\,{\rm (b)} follows once we show 
Proposition \ref{P4.2-2}\,{\rm (a)}.

For $(w,\varepsilon)\in\mathfrak{I}_n(p_n,q_n)$
we set
\begin{equation}
\label{4.5.4}
\hat{w}(\ell)=\left\{
\begin{array}{ll}
\tau w\tau^{-1}(\ell) & \mbox{if $\ell\leq n$,} \\
\ell & \mbox{if $\ell\geq n+1$}
\end{array}
\right.\quad\mbox{for all $\ell\in\mathbb{N}^*$,}
\end{equation}
where $\tau=\tau^{(n)}\in\mathfrak{S}_n$ is as in (\ref{completed-basis}), and
\begin{equation}
\label{4.5.5}\hat{\varepsilon}(\ell)=\left\{
\begin{array}{ll}
\varepsilon\tau^{-1}(\ell) & \mbox{if $\ell\leq n$,} \\
1 & \mbox{if $\ell\geq n+1$, $n\in N_+$,} \\
-1 & \mbox{if $\ell\geq n+1$, $n\in N_-$}
\end{array}
\right.
\end{equation}
for all $\ell\in\mathbb{N}^*$ such that $\hat{w}_\ell=\ell$.
We have readily seen that $(\hat{w},\hat{\varepsilon})\in\mathfrak{I}_\infty(N_+,N_-)$,
and in fact the so obtained map $j_n:\mathfrak{I}_n(p_n,q_n)\to\mathfrak{I}_\infty(N_+,N_-)$ is well defined, injective,
and 
\[
\mathfrak{I}_\infty(N_+,N_-)=\bigcup_{n\geq 1}j_n(\mathfrak{I}_n(p_n,q_n)).
\]
Moreover, it follows from our constructions that, given a basis $\underline{v}=(v_1,\ldots,v_n)$ of $V_n$ and the basis $\underline{\hat{v}}$ of $\mathbf{V}$ obtained by adding the vectors $e_\ell$ for $\ell\geq n+1$, we have:
\begin{eqnarray*}
 & \mbox{$(v_{\tau_1},\ldots,v_{\tau_n})$ is $(w,\varepsilon)$-conjugate (resp., dual)} \\
 & \Rightarrow \
\mbox{$\underline{\hat{v}}$ is $(\hat{w},\hat{\varepsilon})$-conjugate (resp., dual)}.
\end{eqnarray*}
As in the proof of Proposition \ref{P4-1.1} we derive the equalities
\begin{equation}
\label{T-proof.2}
\bs{\mathcal{O}}_{(\hat{w},\hat{\varepsilon})}\cap X_{n}=\mathcal{O}_{(w,\varepsilon)}
\quad
\mbox{and}
\quad
\bs{\mathfrak{O}}_{(\hat{w},\hat{\varepsilon})}\cap X_{n}=\mathfrak{O}_{(w,\varepsilon)}
\end{equation}
where $\mathcal{O}_{(w,\varepsilon)},\mathfrak{O}_{(w,\varepsilon)}\subset X_{n}$ are as in Definition \ref{DA3}.
Parts {\rm (a)} and {\rm (c)} of Proposition \ref{P4.2-2} then follow
from Proposition \ref{P2}\,{\rm (a)} and {\rm (c)}.
\end{proof}

\begin{proof}[Proof of Proposition \ref{P4.3}]
Let $n\in\{n_1,n_2,\ldots\}$ (where $n_k=|J_1|+\ldots+|J_k|$ as before) and $(p_n,q_n)=(|N_+\cap\{1,\ldots,n\}|,|N_-\cap\{1,\ldots,n\}|)$
and let $\tau=\tau^{(n)}:\{1,\ldots,n\}\to\{1,\ldots,n\}$
be the permutation such that $\sigma(\tau_1)\prec\ldots\prec\sigma(\tau_n)$.
Since the generalized flag $\mathcal{F}_\sigma$ is $\omega$-isotropic,
we must have
\[\iota(\tau_\ell)=\tau_{n-\ell+1}\quad\mbox{for all $\ell\in\{1,\ldots,n\}$}.\]
This observation easily implies that the map $j_n$ defined in the proof 
of Proposition \ref{P4.2-2}
restricts to a well-defined injective map 
\[j_n:\mathfrak{I}^{\eta,\epsilon}_n(p_n,q_n)\to \mathfrak{I}^{\eta,\epsilon}_\infty(N_+,N_-)\]
such that
\[\mathfrak{I}^{\eta,\epsilon}_\infty(N_+,N_-)=\bigcup_{k\geq 1}
j_{n_k}(\mathfrak{I}^{\eta,\epsilon}_{n_k}(p_{n_k},q_{n_k})).\]
By (\ref{T-proof.2}) for $(\hat{w},\hat{\varepsilon})=j_n(w,\varepsilon)$ we get
\begin{equation}
\label{T-proof.3}
\bs{\mathcal{O}}^{\eta,\epsilon}_{(\hat{w},\hat{\varepsilon})}\cap (X_{n})_\omega=\mathcal{O}^{\eta,\epsilon}_{(w,\varepsilon)}
\quad
\mbox{and}
\quad
\bs{\mathfrak{O}}^{\eta,\epsilon}_{(\hat{w},\hat{\varepsilon})}\cap (X_{n})_\omega=\mathfrak{O}^{\eta,\epsilon}_{(w,\varepsilon)}.
\end{equation}
Proposition \ref{P4.3} easily follows from this fact and Proposition \ref{P3}.
\end{proof}

\section{Corollaries}

\label{S5}


We start by a corollary stating that the parametrization of $\mathbf{K}$- and $\mathbf{G}^0$-orbits on $\mathbf{G}/\mathbf{B}$ depends only on the triple $\left(\mathbf{G},\mathbf{K},\mathbf{G}^0\right)$ but not on the choice of the ind-variety $\mathbf{G}/\mathbf{B}$.

\begin{corollary}
\label{C1}
Let $E,\mathbf{G},\mathbf{K},\mathbf{G}^0$ be as in Section \ref{S2.1}.
Let $\mathcal{F}_{\sigma_j}$ ($j=1,2$) be two 
$E$-compatible
maximal generalized flags, which are $\omega$-isotropic in types B,C,D,
and let $\mathbf{X}_j=\mathbf{G}/\mathbf{B}_{\mathcal{F}_{\sigma_j}}$.
Then there are natural bijections 
\[
\mathbf{X}_1/\mathbf{K}\cong\mathbf{X}_2/\mathbf{K}
\quad\mbox{and}\quad
\mathbf{X}_1/\mathbf{G}^0\cong\mathbf{X}_2/\mathbf{G}^0
\]
which commute with the duality of Theorem \ref{theorem-1}.
\end{corollary}

Next, a straightforward counting of the parameters yields:

\begin{corollary}
\label{C2}
In Corollary \ref{C1}
the orbit sets $\mathbf{X}_j/\mathbf{K}$ and $\mathbf{X}_j/\mathbf{G}^0$
are always infinite.
\end{corollary}


It is important to note that, despite Corollary \ref{C1}, the topological properties of the orbits on $\mathbf{G}/\mathbf{B}$ are not the same for different choices of Borel ind-subgroups $\mathbf{B}\subset\mathbf{G}$.  The following corollary establishes criteria for the existence of open and closed orbits on $\mathbf{G}/\mathbf{B}=\mathbf{X}\left(\mathcal{F}_\sigma,E\right)$.

\begin{corollary}
\label{C3}
Let $E,\mathbf{G},\mathbf{K},\mathbf{G}^0$ be as in Section \ref{S2.1},
and let $\mathcal{F}_\sigma$ be an $E$-compatible maximal generalized flag, $\omega$-isotropic in types B,C,D, where $\sigma:\mathbb{N}^*\to (A,\prec)$ is a bijection onto a totally ordered set.
Let $\mathbf{X}=\mathbf{G}/\mathbf{B}_{\mathcal{F}_\sigma}$; i.e., $\mathbf{X}=\mathbf{X}(\mathcal{F}_\sigma,E)$ in type A and $\mathbf{X}=\mathbf{X}_\omega(\mathcal{F}_\sigma,E)$ in types B,C,D.

\begin{itemize}
\item[\rm ($\mbox{a}_1$)] In type A1, 
$\mathbf{X}$ has an open $\mathbf{K}$-orbit (equivalently, a closed $\mathbf{G}^0$-orbit) if and only if $\iota(\ell)=\ell$ for all $\ell\gg 1$ (i.e., if the matrix $\Omega$ of (\ref{omega}) contains finitely many diagonal blocks of size 2).
\item[\rm ($\mbox{a}_2$)]
In type A2, $\mathbf{X}$ has an open $\mathbf{K}$-orbit (equivalently, a closed $\mathbf{G}^0$-orbit) if and only if
for all $\ell\gg 1$ the elements
$\sigma(2\ell-1),\sigma(2\ell)$ are consecutive in $A$ and 
the number $|\{k<2\ell-1:\sigma(k)\prec\sigma(2\ell-1)\}|$ is even.
\item[\rm ($\mbox{a}'_{12}$)]
In types A1 and A2, 
$\mathbf{X}$ has at most one closed $\mathbf{K}$-orbit
(equivalently, at most one open $\mathbf{G}^0$-orbit).  $\mathbf{X}$ has a closed $\mathbf{K}$-orbit (equivalently an open $\mathbf{G}^0$-orbit) if and only if $\mathbf{X}$ contains $\omega$-isotropic generalized flags. This latter condition is equivalent to the existence of an 
involutive antiautomorphism of ordered sets $\iota_A:(A,\prec)\to(A,\prec)$ such that
$\iota_A\sigma(\ell)=\sigma\iota(\ell)$ for all $\ell\gg 1$.
\item[\rm ($\mbox{a}_3$)] In type A3,
$\mathbf{X}$ has always infinitely many closed $\mathbf{K}$-orbits (equivalently, infinitely many open $\mathbf{G}^0$-orbits).
$\mathbf{X}$ has an open $\mathbf{K}$-orbit (equivalently, a closed $\mathbf{G}^0$-orbit)
if and only if $d:=\min\{|N_+|,|N_-|\}<\infty$ and $\mathcal{F}_\sigma$
contains a $d$-dimensional and a $d$-codimensional subspace.
\item[\rm (bcd)] In types B,C,D,
$\mathbf{X}$ has always infinitely many closed $\mathbf{K}$-orbits (equivalently, open $\mathbf{G}^0$-orbits).
In types C1 and D3, $\mathbf{X}$ has never an open $\mathbf{K}$-orbit (equivalently, no closed $\mathbf{G}^0$-orbit).
In types BD1 and C2, $\mathbf{X}$ has an open $\mathbf{K}$-orbit (equivalently, a closed $\mathbf{G}^0$-orbit)
if and only if $d:=\min\{|N_+|,|N_-|\}<\infty$ and $\mathcal{F}_\sigma$ has a $d$-dimensional subspace (or equivalently it has a $d$-codimensional subspace).
\end{itemize}
\end{corollary}

\begin{proof}
This follows from Remarks \ref{R-open} and \ref{R-closed}, Propositions \ref{P4-1.1}, \ref{P4.2-2}, \ref{P4.3}, and relations (\ref{T-proof.1}), (\ref{T-proof.2}), (\ref{T-proof.3}).
\end{proof}


\begin{corollary}
The only situation where $\mathbf{X}$ has simultaneously open and closed $\mathbf{K}$-orbits (equivalently, open and closed $\mathbf{G}^0$-orbits) is in types
A3, BD1, C2, in the case where $d:=\min\{|N_+|,|N_-|\}<\infty$ and $\mathcal{F}_\sigma$
contains a $d$-dimensional and a $d$-codimensional subspace.
\end{corollary}

\section*{Index of notation}

\begin{itemize}
\item[\S\ref{S.intro}:] $\mathbb{N}^*$, $|A|$, $\mathfrak{S}_n$, $\mathfrak{S}_\infty$, $(k;\ell)$

\smallskip
\item[\S\ref{S2.1}:] $\mathbf{G}(E)$, $\mathbf{G}(E,\omega)$, $\Omega$, $\omega$, $\gamma$, $\Phi$, $\phi$, $\delta$

\smallskip
\item[\S\ref{S2.2}:] $\mathcal{F}(v_1,\ldots,v_n)$, $\mathbb{O}_w$

\smallskip
\item[\S\ref{S2.3}:] $\mathcal{F}_\sigma(\underline{v})$, $\mathcal{F}_\sigma$,
$\mathbf{P}_\mathcal{F}$, $\mathbf{B}_\mathcal{F}$, $\mathbf{X}(\mathcal{F},E)$, $(\pmb{\mathbb{O}}_{\tau,\sigma})_w$, $\mathbf{X}_\omega(\mathcal{F},E)$

\smallskip
\item[\S\ref{S3.1}:] $\mathcal{F}^\perp$, $\gamma(\mathcal{F})$, $\mathfrak{I}_n$, $\mathfrak{I}'_n$, $\mathcal{O}_w$, $\mathfrak{O}_w$

\smallskip
\item[\S\ref{S3.2}:] $\delta(\mathcal{F})$, $\mathcal{F}^\dag$, $\varsigma(\phi:\mathcal{F})$, $\varsigma(\delta:\mathcal{F})$, $\varsigma(w,\varepsilon)$, $\mathfrak{I}_n(p,q)$, $\mathcal{O}_{(w,\varepsilon)}$, $\mathfrak{O}_{(w,\varepsilon)}$

\smallskip
\item[\S\ref{S3.3}:] $\mathfrak{I}^{\eta,\epsilon}_n(p,q)$, $\mathcal{O}^{\eta,\epsilon}_{(w,\varepsilon)}$, $\mathfrak{O}^{\eta,\epsilon}_{(w,\varepsilon)}$
\smallskip
\item[\S\ref{S4.1}:] $\iota$, $\mathfrak{I}_\infty(\iota)$, $\mathfrak{I}'_\infty(\iota)$, $\bs{\mathcal{O}}_w$, $\bs{\mathfrak{O}}_w$, $(A^*,\prec^*)$, $\sigma^\perp$, $\mathbf{X}^\perp$, $\mathbf{X}^\gamma$, $\pmb{\mathbb{O}}_w^\perp$, $\pmb{\mathbb{O}}_w^\gamma$
\smallskip
\item[\S\ref{S4.2}:] $\mathfrak{I}_\infty(N_+,N_-)$, $\bs{\mathcal{O}}_{(w,\varepsilon)}$, $\bs{\mathfrak{O}}_{(w,\varepsilon)}$, $\sigma^\dag$, $\mathbf{X}^\dag$, $\pmb{\mathbb{O}}_w$, $\pmb{\mathbb{O}}_w^\dag$
\smallskip
\item[\S\ref{S4.3}:] $\mathfrak{I}_\infty^{\eta,\epsilon}(N_+,N_-)$, $\bs{\mathcal{O}}_{(w,\varepsilon)}^{\eta,\epsilon}$, $\bs{\mathfrak{O}}_{(w,\varepsilon)}^{\eta,\epsilon}$ 
\end{itemize}

\end{document}